\documentclass[11pt,reqno]{amsart}
\usepackage{amssymb,amsfonts,amsmath,bm}
\usepackage{hyperref}
\usepackage{psfrag,graphicx,color}
\usepackage{mathrsfs}
\usepackage{graphics}
\usepackage{subfigure}
\usepackage{enumerate,enumitem}
\usepackage{pgfplots}
\usepackage{cite}
\usepackage{diagbox}

\allowdisplaybreaks

\DeclareMathOperator{\sinc}{sinc}

\usepackage{booktabs}
\usepackage{multirow}

\usepackage[margin=1.2in]{geometry}

\DeclareGraphicsExtensions{.eps}
\usepackage{amssymb,amsmath,graphicx,amsfonts,euscript,mathrsfs}
\usepackage{hyperref}

\usepackage{titlesec}
\titleformat{\section}{\vskip10pt\large\bfseries}{\thesection.}{0.5em}{\centering\vspace{5pt}}
\titleformat{\subsection}{\vskip10pt\normalsize\bfseries}{\thesubsection.}{0.5em}{}

\newtheorem{theorem}{Theorem}[section]

\newtheorem{remark}[theorem]{Remark}
\theoremstyle{definition}

\def\R{{\mathbb R}}

\def\d{{\mathrm d}}

\numberwithin{equation}{section}

\begin{document}

\title[]{A second-order low-regularity correction of Lie splitting for the semilinear Klein--Gordon equation
}

\author[]{Buyang Li}
\address{\hspace*{-12pt}Buyang Li:
Department of Applied Mathematics, The Hong Kong Polytechnic University,
Hung Hom, Hong Kong. {\it E-mail address}: {\tt buyang.li@polyu.edu.hk} }

\author[]{Katharina Schratz}
\address{\hspace*{-12pt}Katharina Schratz and Franco Zivcovich:
Laboratoire Jacques-Louis Lions, Sorbonne Université,
Bureau : 16-26-315, 4 place Jussieu, Paris 5ème. \newline
{\it E-mail address}: {\tt katharina.schratz@sorbonne-universite.fr}\, and\,
{\tt franco.zivcovich@gmail.com} }

\author[]{Franco Zivcovich}





\maketitle
\begin{abstract}\noindent
The numerical approximation of the semilinear Klein--Gordon equation in the $d$-dimensional space, with $d=1,2,3$, is studied by analyzing the consistency errors in approximating the solution. By discovering and utilizing a new cancellation structure in the semilinear Klein--Gordon equation, a low-regularity correction of the Lie splitting method is constructed, which can have second-order convergence in the energy space under the regularity condition $(u,\partial_tu)\in L^\infty(0,T;H^{1+\frac{d}{4}}\times H^{\frac{d}{4}})$, where $d=1,2,3$ denotes the dimension of space.
In one dimension, the proposed method is shown to have a convergence order arbitrarily close to $\frac53$ in the energy space for solutions in the same space, i.e. no additional regularity in the solution is required.
Rigorous error estimates are presented for a fully discrete spectral method with the proposed low-regularity time-stepping scheme.
Numerical examples are provided to support the theoretical analysis and to illustrate the performance of the proposed method in approximating both nonsmooth and smooth solutions of the semilinear Klein--Gordon equation.
\end{abstract}



\section{\bf Introduction}\label{sec:intr}


We consider the following initial-boundary value problem of the semilinear Klein--Gordon equation:
\begin{equation}
\label{pde}
\left \{
\begin{aligned}
&\partial_{tt} u - \Delta u = f(u) && \mbox{in}\,\,\,  \varOmega\times (0,T] ,\\
&u=0 && \mbox{on}\,\,\, \partial\varOmega\times (0,T] ,\\
&u|_{t=0}=u^0\,\,\,\mbox{and}\,\,\,
\partial_tu|_{t=0}=v^0 && \mbox{in}\,\,\, \varOmega ,
\end{aligned}
\right .
\end{equation}
in a rectangular domain $\Omega\subset\R^d$ under the homogeneous Dirichlet boundary condition, where $f:\R\rightarrow\R$ is a given nonlinear function. For example,
\eqref{pde} is often referred to as the sine--Gordon equation in the case $f(u)=\sin(u)$, which arises in many physical applications, such as magnetic-flux propagation in Josephson junctions, bloch-wall dynamics in magnetic crystals, propagation of dislocation in solid and liquid crystals, propagation of ultra-short optical pulses in two-level media; see \cite{BEM1971}.
Since the semilinear Klein--Gordon equation describes wave propagation with finite speed,
when the initial values $u_0$ and $v_0$ have compact supports,
problem \eqref{pde} can also be used to describe wave propagation in the whole space $\R^d$ by choosing a sufficiently large $\Omega$ (so that the wave does not reach the boundary up to time $T$).



The numerical approximation of semilinear Klein--Gordon equations in the form of \eqref{pde} has been extensively studied in computational mathematics. A large variety of numerical schemes for approximating the time dynamics of the semilinear Klein--Gordon equation has been proposed and analyzed, including trigonometric/exponential integrators that are based on the variation-of-constants formula (for example, see
\cite{BDH2021,GH06,HL1999,G15,WW2019}), splitting methods (for example, see \cite{BFS2022,BY2007,BDH2021,HO2016}), finite difference methods (such as the Crank–Nicolson and Runge--Kutta methods, see \cite{CLZ2021,HL2021,LLT2016,LS2020,LV2006,MK2011,QW2019}), and symplectic methods \cite{CHSS2020,CHL2008,HL00}.

The analyses in these articles (for example in \cite{BDH2021,HL1999,G15,WW2019}) have shown that for initial data $(u^0,v^0)$ in the physically natural energy space $H^1(\Omega)\times L^2(\Omega)$, so that the solution $(u,\partial_tu)$ is bounded in the energy space $H^1(\Omega)\times L^2(\Omega)$ uniformly for $t\in[0,T]$, the classical time-stepping methods such as splitting methods, Runge--Kutta methods, trigonometric integrators, and averaged exponential integrators, can approximate the solution $(u,\partial_tu)$ with second-order convergence in the weaker space $L^2(\Omega)\times H^{-1}(\Omega)$, but only with first-order convergence in the energy space $H^1(\Omega)\times L^2(\Omega)$ itself. Moreover, the second-order approximation to $(u,\partial_tu)$ in the energy space $H^1(\Omega)\times L^2(\Omega)$ generally requires the initial data to be in the stronger space  $H^2(\Omega)\times H^1(\Omega)$.

The only method which breaks this order barrier is the low-regularity integrator proposed in \cite{RS21}, which can have second-order convergence in the energy space $H^1(\Omega)\times L^2(\Omega)$ under the weaker regularity condition $(u^0,v^0)\in H^{\frac74}(\Omega)\times H^{\frac34}(\Omega)$; see \cite[Corollary 5.7]{RS21}. This low-regularity integrator is based on the reformulation of \eqref{pde} into the first-order equation
\begin{align}
\label{formula:rootrewriting}
i \partial_{t} w=-(-\Delta)^{\frac12} w+(-\Delta)^{-\frac12} f\Big(\frac{w+\bar{w}}{2}\Big)
\end{align}
through the transformation $w=u-i(-\Delta)^{-\frac12} \partial_{t} u$, which is then discretized by the low-regularity integrators proposed in \cite{RS21} for first-order semilinear evolution equations.
Such low-regularity types of numerical schemes have recently gained a lot of attention in particular in the context of the nonlinear Schr\"odinger equation (see, e.g., \cite{BS,OS18,ORS20,RS21}), KdV equation (see, e.g., \cite{Hofmanova-Schratz-2017,Li-Wu-Yao-2021,Wu-Zhao-IMA,Wu-Zhao-BIT}), and the Navier--Stokes equations \cite{LMS}.
Second-order approximations to the solutions of these equations in the $H^s$ norm generally require the solutions to be bounded in $H^s(\Omega)$ for $s>d/2+1$.

In this article, we construct a new time-stepping method for the semilinear Klein--Gordon equation through analyzing the consistency errors in approximating the solution. By discovering and utilizing a new cancellation structure of the semilienar Klein--Gordon equation,we manage to find a low-regularity correction of the Lie splitting method, i.e.,
\begin{align}\label{2nd-method}
\left(\begin{array}{c}
u^{n+1} \\
v^{n+1}
\end{array}\right)=\underbrace{e^{\tau L}\left(\begin{array}{c}
u^{n} \\
v^{n}
\end{array}\right)+\tau e^{\tau L}\left(\begin{array}{c}
0 \\
f\left(u^{n}\right)
\end{array}\right)}_{\text {Lie splitting }}+\underbrace{\tau^{2} e^{\tau L} \varphi_{2}(-2 \tau L)\left(\begin{array}{c}
-f\left(u^{n}\right) \\
f^{\prime}\left(u^{n}\right) v^{n}
\end{array}\right)}_{\text {low-regularity correction}}
\end{align}
where $(u^n,v^n)^\top$ is an approximation to $(u(t_n),\partial_tu(t_n))^\top$, and $L$ is a linear anti-symmetric partial differential operator defined by
\begin{equation}\label{def-L}
L =
\left(\begin{array}{cc}
0 & 1 \\
\Delta & 0
\end{array} \right) :
[H^2(\Omega)\cap H^1_0(\Omega)]\times H^1(\Omega)\rightarrow H^1_0(\Omega)\times L^2(\Omega) .
\end{equation}
The last term in \eqref{2nd-method}, which contains the operator $\varphi_2(-2\tau L):=(2 \tau L)^{-2}\left(e^{-2 \tau L}+2 \tau L-I\right) $,
is a low-regularity correction term for the Lie splitting, i.e. it improves the Lie splitting method to second order under low-regularity conditions, without requiring second-order partial derivatives of the solution.
Theoretically, we prove that the new time-stepping method can achieve second-order convergence in the energy space $H^1(\Omega)\times L^2(\Omega)$ under the regularity condition $(u^0,v^0)\in H^{1+\frac{d}{4}}(\Omega)\times H^{\frac{d}{4}}(\Omega)$ for spatial dimension $d=1,2,3$; see Theorem \ref{THM:2nd-H1}. In the one-dimension case, the proposed method is shown to have a convergence order arbitrarily close to $\frac53$ in the energy space $H^1(\Omega)\times L^2(\Omega)$ for solutions in the same space, i.e. no additional regularity in the solution is required.
The numerical experiments in this article shows that the proposed method is practically higher-order than all the existing numerical methods for the semilinear Klein--Gordon equation for low-regularity solutions in the energy space $H^1(\Omega)\times L^2(\Omega)$.


%
%


The following convergence result is proved in this article.

\begin{theorem}\label{THM:2nd-H1-TD}
Let $f:\R\rightarrow\R$ be a given nonlinear function satisfying the following Lipschitz continuity condition
{\rm(}for some constants $C_1${\rm)}:
\begin{align}\label{condition-f}
|f'(s)|\le C_1 \quad\mbox{and}\quad |f''(s)|\le C_1 \quad\mbox{for}\,\,\, s\in\R .
\end{align}
Then, for $d=1,2,3$ and $(u^0,v^0) \in H^{1+\frac{d}{4}}(\Omega)\cap H^1_0(\Omega)\times H^{\frac{d}{4}}(\Omega)$, the numerical solution given by \eqref{2nd-method} has the following error bound:
\begin{align}
\max_{0\le n\le T/\tau}
\big( \|u^n-u(t_n)\|_{H^1(\Omega)} + \|v^n-\partial_tu(t_n)\|_{L^2(\Omega)} \big)
\le
&C_2 \tau^2 ,
\label{H1-error}
\end{align}
where $C_2$ is some positive constant independent of the stepsize $\tau$
{\rm(}but may depend on $T${\rm)}.

Moreover, for $d=1$ and $(u^0,v^0) \in H^1_0(\Omega)\times L^2(\Omega)$, the numerical solution given by \eqref{2nd-method} has the following error bound:
\begin{align}
\max_{0\le n\le T/\tau}
\big( \|u^n-u(t_n)\|_{H^1(\Omega)} + \|v^n-\partial_tu(t_n)\|_{L^2(\Omega)} \big)
\le
&C_3 \tau^{\frac53-\epsilon} ,
\label{H1-error22}
\end{align}
where $\epsilon\in(0,1)$ is an arbitrary fixed small constant, and $C_3$ is some positive constant independent of the stepsize $\tau$
{\rm(}but may depend on $T${\rm)}.
\end{theorem}

\begin{remark}\upshape
The consistency errors of the numerical method actually only contain first-order partial derivatives of the solution, instead of $1+\frac{d}{4}$ order partial derivatives. The regularity condition $H^{1+\frac{d}{4}}(\Omega)\times H^{\frac{d}{4}}(\Omega)$ arises from the use of Sobolev embedding $H^{1+\frac{d}{4}}(\Omega)\hookrightarrow W^{1,4}(\Omega)$ in the error estimation. In the numerical experiments (see Figure \ref{fig:sineKG_th1} in Section \ref{sec:numExp}), we observe second-order convergence of the method for $H^1(\Omega)\times L^2(\Omega)$ initial data.
\end{remark}

The Lipschitz continuity condition in \eqref{condition-f} can be removed in the case $d=1$, as the $L^\infty$ bound of the numerical solution $u^n$ can be proved by using its convergence in $H^1$. For $d=2,3$ this Lipschitz continuity condition is needed for a general nonlinear function $f(u)$, but is still possible to be removed for some special nonlinear functions such as $f(u)=u^2$. Since such analysis requires different treatments for different nonlinearities (for $d=2,3$), we focus on the construction of the low-regularity integrator in the general case $d=1,2,3$ with a general nonlinear function under the Lipschitz continuity condition.

The rest of this article is devoted to the construction of the method and the proof of the theorem.
In Section \ref{sec:main_results} we construct the second-order low-regularity integrator by analyzing the consistency errors in approximating the semilinear Klein--Gordon equation. In Section \ref{section:spatial} we present error estimates for a fully discrete spectral method with the time-stepping scheme in \eqref{2nd-method} (see Theorem \ref{THM:2nd-H1} and Remark \ref{remarkTHM}), which imply Theorem \ref{THM:2nd-H1-TD} by passing to the limit $N\rightarrow\infty$, where $N^d$ denotes the degrees of freedom in the spatial discretization.
The numerical experiments are presented in Section \ref{sec:numExp} to show the favorable error behaviour of the new scheme for both nonsmooth and smooth initial data.

\section{Construction of the low-regularity integrators}\label{sec:main_results}
We rewrite the semilinear Klein--Gordon equation into the following first-order system, i.e.,
\begin{equation}
\label{pde2}
\left \{
\begin{aligned}
&\partial_t U - L U = F(U) && \mbox{in}\,\,\,  \varOmega\times (0,T] ,\\
&U(t_n)=U^0 && \mbox{in}\,\,\, \varOmega ,
\end{aligned}
\right .
\end{equation}
where
\begin{equation}\label{nonlinearity-F}
U=
\left(\begin{aligned}
u\,\, \\
\partial_tu
\end{aligned} \right)
,\quad
U^0=
\left(\begin{aligned}
u^0 \\
v^0
\end{aligned} \right)
\quad\mbox{and}\quad
F(U) =
\left(\begin{array}{cc}
0 \\
f(u)
\end{array} \right) ,
\end{equation}
and $L$ is defined in \eqref{def-L}.
Under the Lipschitz continuity condition \eqref{condition-f}, it is well known that problem \eqref{pde2} has a unique energy solution $U\in L^\infty(0,T;H^1_0(\Omega)\times L^2(\Omega))$ satisfying the following variation-of-constants formula:
\begin{align}\label{Duhamel-f}
U(t+s) = e^{sL}U(t) + \int_0^s e^{(s-\sigma)L}F(U(t+\sigma))\d \sigma \quad\mbox{for}\,\,\, t,s\ge 0 ,
\end{align}
where $e^{tL}$ is the continuous semigroup on $H^1_0(\Omega)\times L^2(\Omega)$ generated by the anti-symmetric partial differential operator $L$.

For the simplicity of notation, we denote by $A\lesssim B$ the statement
``$A\le CB$ for some constant $C$ which is independent of the stepsize $\tau$
(or the degrees of freedom $N$ in the case there is spatial discretization)''.

For the error analysis we define the energy norm $|W|_{1} = \big( \|\nabla w_1\|_{L^2(\Omega)}^2 + \|w_2\|_{L^2(\Omega)}^2 \big)^{\frac12} $ and the following non-energy norms:
\begin{align*}
\|W\|_{0} &= \big( \|w_1\|_{L^2(\Omega)}^2 + \|w_2\|_{H^{-1}(\Omega)}^2\big)^{\frac12} ,\\
\|W\|_{1} &= \big( \|\nabla w_1\|_{L^2(\Omega)}^2 + \|w_2\|_{L^2(\Omega)}^2 \big)^{\frac12} ,\\
\|W\|_{2} &= \big( \|w_1\|_{H^2(\Omega)}^2 + \|w_2\|_{H^{1}(\Omega)}^2\big)^{\frac12} .
\end{align*}
It is known that the semigroup $e^{tL}$ satisfies the energy conservation $|e^{tL} W |_1 =  |W|_1$ for $W\in H^1_0(\Omega)\times L^2(\Omega)$, and the following estimates:
\begin{align}\label{Hk-estimates}
\begin{aligned}
\|e^{tL} W \|_0 &\lesssim \|W\|_0 \quad\forall\, W\in L^2(\Omega)\times H^{-1}(\Omega) ,\\
\|e^{tL} W \|_1 &\lesssim \|W\|_1 \quad\forall\, W\in H^1_0(\Omega)\times L^2(\Omega) ,\\
\|e^{tL} W \|_2 &\lesssim \|W\|_2 \quad\forall\, W\in [H^2(\Omega)\cap H^1_0(\Omega)] \times L^2(\Omega) .
\end{aligned}
\end{align}
Moreover, the nonlinear function $F(U)$ defined in \eqref{nonlinearity-F} satisfies the following estimate:
\begin{equation}
\|F(U)\|_1 \lesssim \|f(u)\|_{L^2} \lesssim \|U\|_{0} .
\end{equation}

In the following two subsections, we study the consistency errors in approximating formula \eqref{Duhamel-f}. We begin with a first-order approximation in the next subsection, which provides insights for us for the construction of the second-order low-regularity integrator.

\subsection{First-order approximation}
Let $t_n=n\tau$, $n=0,1,\dots,[T/\tau]$, be a sequence of discrete time levels with stepsize $\tau$, and consider the variation-of-constant formula:
\begin{align}\label{Duhamel2}
U(t_n+s) = e^{sL}U(t_n) + \int_0^s e^{(s-\sigma)L}F(U(t_n+\sigma))\d \sigma \quad\mbox{for}\,\,\, s\in[0,\tau] ,
\end{align}
which implies that
\begin{align}\label{Duhamel}
U(t_{n+1}) = e^{\tau L}U(t_n) + \int_{0}^{\tau} e^{(\tau-s)L}F(U(t_n+s))\d s  .
\end{align}
Substituting \eqref{Duhamel2} into the right-hand side of \eqref{Duhamel} yields
\begin{align}\label{Duhamel3}
U(t_{n+1}) =& e^{\tau L}U(t_n) + \int_0^{\tau} e^{(\tau-s)L}F(e^{sL}U(t_n))\d s
+ R_{1}(t_{n}) ,
\end{align}
where the remainder $R_{1}(t_{n})$ is given by
\begin{align}\label{R1}
R_{1}(t_{n})
= & \int_0^\tau e^{(\tau-s)L}[ F(U(t_n+s))-F(e^{sL}U(t_n))] \d s .
\end{align}

For the simplicity of notation, we denote by $\tilde u(t_n+s)$ and $\tilde v(t_n+s)$ the two functions defined by
$$
\left(\begin{array}{cc}
\tilde u(t_n+s)  \\
\tilde v(t_n+s)
\end{array} \right)
= e^{s L}
\left(\begin{array}{cc}
u(t_n) \\
\partial_tu(t_n)
\end{array} \right)
= e^{s L} U(t_n) .
$$
Then the remainder $R_1(t_n)$ defined in \eqref{R1} satisfies the following estimate in view of \eqref{Duhamel2}:
\begin{align*}
\| R_{1}(t_{n}) \|_1
\lesssim & \int_0^\tau \| F(U(t_n+s))-F(e^{sL}U(t_n)) \|_1 \d s \notag\\
= & \int_0^\tau \| f(u(t_n+s)) - f(\tilde u(t_n+s)) \|_{L^2(\Omega)} \d s \notag\\
\lesssim & \int_0^\tau \| u(t_n+s) - \tilde u(t_n+s) \|_{L^2(\Omega)} \d s \notag\\
\lesssim & \int_0^\tau \| U(t_n+s) - e^{sL}U(t_n) \|_0 \d s \notag\\
\lesssim & \int_0^\tau \int_0^s \| e^{(s-\sigma)L}F(U(t_n+\sigma)) \|_0 \d \sigma \d s \notag\\
\lesssim & \int_0^\tau \int_0^s \|F(U(t_n+\sigma)) \|_0 \d \sigma \d s \notag\\
\lesssim &\, \tau^2 \max_{\sigma\in[0,\tau]} \|f(u(t_n+\sigma))\|_{H^{-1}} .
\end{align*}
Since $H^1_0(\Omega)\hookrightarrow L^6(\Omega)$, it follows that $L^{\frac65}(\Omega)=L^{6}(\Omega)'\hookrightarrow H^1_0(\Omega)'=H^{-1}(\Omega)$ and therefore
\begin{align*}
\|f(u(t_n+\sigma))\|_{H^{-1}}
&\le \|f(u(t_n+\sigma))\|_{L^{6/5}} \notag\\
&\le \|f(u(t_n+\sigma))\|_{L^2}  \notag\\
&\le \|f(0)\|_{L^2}+\|f(u(t_n+\sigma))-f(0)\|_{L^2} \notag\\
&\le \|f(0)\|_{L^2} + \|u(t_n+\sigma)\|_{L^2} \notag\\
&\le \|f(0)\|_{L^2} + \|U(t_n+\sigma)\|_{0} .
\end{align*}
The two estimates above imply the following estimate for the remainder $R_{1}(t_{n})$:
\begin{align}\label{R1-estimate}
\| R_{1}(t_{n}) \|_1
\lesssim &\, \tau^2\Big(1+ \max_{\sigma\in[t_n,t_{n+1}]} \|U(t)\|_{0}\Big) .
\end{align}

Freezing the variable $s$ at $0$ in \eqref{Duhamel3} would yield
\begin{align}\label{Duhamel32}
U(t_{n+1}) = e^{\tau L}U(t_n) +  \tau e^{\tau L}F(U(t_n))
+ R_{1}(t_n)+ R_{2}(t_n)  ,
\end{align}
with an additional remainder
\begin{align}\label{R2}
R_{2}(t_n)
= & \int_0^\tau e^{\tau L} [e^{-sL}F(e^{sL}U(t_n)) - F(U(t_n)) ] \d s \notag \\
= & \int_0^\tau e^{\tau L} \int_0^s \frac{\d}{\d \sigma}e^{-\sigma L}F(e^{\sigma L}U(t_n))\d \sigma \d s .
\end{align}
By using the chain rule of differentiation, it is straightforward to verify that
\begin{align}\label{R2-estimate1}
 \frac{\d}{\d \sigma}e^{-\sigma L}F(e^{\sigma L}U(t_n))
= & \frac{\d}{\d \sigma}e^{-\sigma L}F(e^{\sigma L}U(t_n)) \notag \\
= & - e^{-\sigma L} L F(e^{\sigma L}U(t_n))
+ e^{-\sigma L} F'(e^{\sigma L}U(t_n)) e^{\sigma L}LU(t_n) \notag \\
= & - e^{-\sigma L}
\left(\begin{array}{cc}
0 & 1 \\
\Delta & 0
\end{array} \right)
\left(\begin{array}{cc}
0 \\
f(\tilde u(t_n+\sigma))
\end{array} \right) \notag\\
&
+ e^{-\sigma L}
\left(\begin{array}{cc}
0 & 0 \\
f'(\tilde u(t_n+\sigma)) & 0
\end{array} \right)
\bigg[
\left(\begin{array}{cc}
0 & 1 \\
\Delta & 0
\end{array} \right)
e^{\sigma L} \left(\begin{array}{cc}
u(t_n)\\
\partial_t u(t_n)
\end{array} \right) \bigg] \notag\\
= & - e^{-\sigma L}
\left(\begin{array}{cc}
0 & 1 \\
\Delta & 0
\end{array} \right)
\left(\begin{array}{cc}
0 \\
f(\tilde u(t_n+\sigma))
\end{array} \right) \notag\\
&
+ e^{-\sigma L}
\left(\begin{array}{cc}
0 & 0 \\
f'(\tilde u(t_n+\sigma)) & 0
\end{array} \right)
 \left(\begin{array}{cc}
\tilde v(t_n+\sigma) \\
\Delta \tilde u(t_n+\sigma)
\end{array} \right)  \notag\\
=  &\,
e^{-\sigma L}
\left(\begin{array}{cc}
-f(\tilde u(t_n+\sigma)) \\
f'(\tilde u(t_n+\sigma)) \tilde v(t_n+\sigma)
\end{array} \right) .
\end{align}
Therefore,
\begin{align}\label{R2-estimate2}
\bigg\| \frac{\d}{\d \sigma}e^{-\sigma L}F(e^{\sigma L}U(t_n)) \bigg\|_1
&\lesssim
\left\| \left(\begin{array}{cc}
-f(\tilde u(t_n+\sigma)) \\
f'(\tilde u(t_n+\sigma)) \tilde v(t_n+\sigma)
\end{array} \right)  \right\|_1 \notag\\
&\lesssim
(1+\| \tilde u(t_n+\sigma) \|_{H^1(\Omega)})
+  \|\tilde v(t_n+\sigma) \|_{L^2(\Omega)} \notag\\
&\lesssim
1+\|  U(t_n) \|_{1}  .
\end{align}
By utilizing this result, from \eqref{R2} we obtain
\begin{align}\label{R2-estimate4}
\|R_{2}(t_n) \|
\lesssim
 \tau^2 (1+ \|  U(t_n) \|_{1}) .
\end{align}

By dropping the remainders $R_{1} $ and $R_{2}$ in \eqref{Duhamel32}, we obtain the following time-stepping method:
\begin{align}\label{semidiscrete}
U^{n+1} =& e^{\tau L} U^{n} +  \tau  e^{\tau L}F(U^{n})
\end{align}
In view of the two estimates \eqref{R1-estimate} and \eqref{R2-estimate4}, the method in \eqref{semidiscrete} should have first-order convergence in the energy space $H^1_0(\Omega)\times L^2(\Omega)$ under the regularity condition $$U\in L^\infty(0,T;H^1_0(\Omega)\times L^2(\Omega)). $$

This is the same regularity condition in \cite{HL1999,G15,WW2019} for first-order convergence in the energy space. This condition is required in \eqref{R2-estimate2} in estimating the remainder $R_2(t_n)$, which involves $\frac{\d}{\d \sigma}e^{-\sigma L}F(e^{\sigma L}U(t_n))$.

From the analysis above we can see that, in order to have higher-order convergence in the energy space, higher-order approximations of $F(U(t_n+s))$ should be used in approximating \eqref{Duhamel}. This is considered in the next subsection.

In the construction of a second-order method, the remainder which involves the term $\frac{\d}{\d \sigma}e^{-\sigma L}F(e^{\sigma L}U(t_n))$ will require the solution to be in $H^2(\Omega)\times H^1(\Omega)$. We shall construct a second-order approximation by eliminating this part of the remainder, thus significantly improves the order of convergence without requiring additional regularity of the solution.

\subsection{Second-order approximation} \label{section:2nd}
By using the Taylor expansion of $F(U)$ at $U=e^{sL}U(t_n)$, we have
\begin{align}\label{Taylor-F1}
F(U(t_n+s))
=& \,
F(e^{sL}U(t_n))
+\int_0^1 F'((1-\theta) e^{sL}U(t_n) +\theta U(t_n+s)) (U(t_n+s)-e^{sL}U(t_n)) \d\theta \notag\\
=& \,
F(e^{sL}U(t_n))
+F'(e^{sL}U(t_n))(U(t_n+s)-e^{sL}U(t_n)) \notag\\
&\,
+ R_F(s)  (U(t_n+s)-e^{sL}U(t_n))
\cdot (U(t_n+s)-e^{sL}U(t_n))
\end{align}
where
$$
R_F(s)
=  \int_0^1 \int_0^1  \theta F''[(1-\sigma)e^{sL}U(t_n)  + \sigma (1-\theta) e^{sL}U(t_n) +\theta U(t_n+s)] \d\sigma\d\theta .
$$
Then, substituting \eqref{Duhamel2} into \eqref{Taylor-F1}, we have
\begin{align}\label{Taylor-F}
F(U(t_n+s))
=& \,
F(e^{sL}U(t_n))
+F'(e^{sL}U(t_n)) \int_0^s e^{(s-\sigma)L}F(U(t_n+\sigma))\d \sigma + \widetilde R_3(s) ,
\end{align}
where
$$
\widetilde R_3(s) = R_F(s) \int_0^s e^{(s-\sigma)L}F(U(t_n+\sigma))\d \sigma
\cdot \int_0^s e^{(s-\sigma)L}F(U(t_n+\sigma))\d \sigma .
$$
Since $F(U)=(F_1(U),F_2(U))^\top$ is vector-valued, with $F_1(U)=0$ and $F_2(U)=f(u)$,
it follows that $F''(U)$ is tensor-valued and satisfying
$F''_{ijk}(U)=\partial_{U_k}\partial_{U_j}F_{i}(U)$, where $U_1=u$ and $U_2=v$.
In particular, $F_{211}''(U)=f''(u)$ and $F_{ijk}''(U)=0$ for $(i,j,k)\neq (2,1,1)$.
Therefore, for $W=(w_1,w_2)^\top$ and $W^*=(w_1^*,w_2^*)^\top$,
$$
\|R_F(s) W\cdot W^*\|_1
\le \| f''(u)w_1w_1^* \|_{L^2}
\lesssim \|w_1 \|_{L^4} \|w_1^*\|_{L^4}
\lesssim \|W \|_{1}\|W^* \|_{1}
 ,
$$
which implies the following estimate:
\begin{align}\label{tR31}
\| \widetilde R_3(s)\|_1
&\lesssim
\bigg\|\int_0^s e^{(s-\sigma)L}F(U(t_n+\sigma))\d \sigma\bigg\|_{1}^2 \notag\\
&\lesssim
\bigg|\int_0^s \| F(U(t_n+\sigma)) \|_{1}\d \sigma\bigg|^2 \notag\\
&\lesssim
\bigg|\int_0^s \| f(u(t_n+\sigma)) \|_{L^2}\d \sigma\bigg|^2 \notag\\
&\lesssim
\tau^2 \Big( 1 + \max_{\sigma\in[0,s]} \| u(t_n+\sigma) \|_{L^2}^2 \Big) \notag\\
&\lesssim  \tau^2\Big(1+ \max_{\sigma\in[0,\tau]} \|U(t_n+\sigma)\|_{0}\Big) .
\end{align}
By substituting \eqref{Taylor-F} into \eqref{Duhamel}, we obtain
\begin{align}\label{2nd-expr}
U(t_{n+1})
=&\, e^{\tau L}U(t_n) + \int_0^\tau e^{(\tau-s)L}F(U(t_n+s))\d s \notag\\
=&\, e^{\tau L}U(t_n) + \int_0^\tau e^{(\tau-s)L}F(e^{sL}U(t_n))\d s \notag\\
&\, + \int_0^\tau  e^{(\tau-s)L} \bigg[ F'(e^{sL}U(t_n))\int_0^s e^{(s-\sigma)L}F(U(t_n+\sigma))\d \sigma\bigg]  \d s  + R_3(t_n) \notag\\
=&\!: e^{\tau L}U(t_n) + I_1(t_n) + I_2(t_n)+ R_3(t_n),
\end{align}
with a remainder
$$
R_3(t_n)= \int_0^\tau  e^{(\tau-s)L} \widetilde R_3(s)\d s.
$$
The estimate in \eqref{tR31} implies the following result:
\begin{align}\label{tR3}
\| R_3(t_n) \|_1
\lesssim
\tau^3 \Big( 1+ \max_{t\in[t_n,t_{n+1}]} \| U(t) \|_{0}^2 \Big) .
\end{align}
The two terms $I_1(t_n)$ and $I_2(t_n)$ will be approximated by computable schemes as follows. \medskip

{\it Part 1: Approximation to $I_1(t_n)$.}

The key ingredient that significantly improves the accuracy of the numerical method is the discovery of a cancellation structure which allows us to compute $I_1(t_n)$ exactly.

We write $I_1(t_n) =  \int_0^\tau e^{\tau L}G(t_n+s)\d s$, with $G(t_n+s)=e^{-sL}F(e^{sL}U(t_n))$, and substitute the Newton--Leibniz formula
\begin{align}\label{Taylor-G}
G(t_n+s)
=G(t_n) + \int_0^s G'(t_n+\sigma)\d\sigma
\end{align}
into the expression of $I_1(t_n)$. Then we obtain
\begin{align}\label{I1-expr10}
I_1(t_n) = &\, \int_0^\tau e^{\tau L}G(t_n+s)\d s \notag\\
= &\, \int_0^\tau e^{\tau L}G(t_n)\d s
+ \int_0^\tau e^{\tau L}\int_0^s G'(t_n+\sigma)\d\sigma \d s \notag\\
= &\, \int_0^\tau e^{\tau L}G(t_n)\d s
+ \int_0^\tau e^{\tau L} G'(t_n+\sigma)(\tau-\sigma)\d\sigma  \notag\\
= &\, \int_0^\tau e^{\tau L}G(t_n)\d s
+ \int_0^\tau e^{(\tau-2s) L}(\tau-s) e^{2sL}G'(t_n+s)  \d s  \notag\\
= &\, \int_0^\tau e^{\tau L}G(t_n)\d s
+ \int_0^\tau e^{(\tau-2s) L}(\tau-s) G'(t_n)  \d s \notag\\
&
+  \int_0^\tau e^{(\tau-2s) L} (\tau-s) \int_0^s \frac{\d}{\d\sigma}[e^{2\sigma L}G'(t_n+\sigma)]  \d\sigma \d s \notag\\
= &\, \tau e^{\tau L}F(U(t_n))
+ (2L)^{-1} \big[ \tau e^{\tau L} - (2L)^{-1}(e^{\tau L} - e^{-\tau L})\big]
\left(\begin{array}{cc}
-f(u(t_n)) \\
f'(u(t_n)) \partial_tu(t_n)
\end{array} \right) \notag\\
&
+ R_*(t_n) ,
\end{align}
where we have used the expression of $G'(t_n+s)$ in \eqref{R2-estimate1}, and the remainder $R_*(t_n)$ is defined by
\begin{align}
R_*(t_n) = \int_0^\tau e^{(\tau-2s) L} (\tau-s) \int_0^s \frac{\d}{\d\sigma}[e^{2\sigma L}G'(t_n+\sigma)]  \d\sigma \d s .
\end{align}
By differentiating $e^{2s L}G'(t_n+s)$ and using the expression of $G'(t_n+s)$ in \eqref{R2-estimate1}, we also obtain
\begin{align} \label{cancellation}
&\hspace{-10pt} \frac{\d}{\d s}[e^{2s L}G'(t_n+s)] \notag\\
=&\,
\frac{\d}{\d s} \bigg[ e^{s L}
\left(\begin{array}{cc}
-f(\tilde u(t_n+s)) \\
f'(\tilde u(t_n+s)) \tilde v(t_n+s)
\end{array} \right) \bigg] \notag\\
=&\,
e^{s L}
\left(\begin{array}{cc}
0 &  1 \\
\Delta & 0
\end{array} \right)
\left(\begin{array}{cc}
-f(\tilde u(t_n+s)) \\
f'(\tilde u(t_n+s)) \tilde v(t_n+s)
\end{array} \right)  \notag\\
&\,
+
e^{s L}
\left(\begin{array}{cc}
-f'(\tilde u(t_n+s)) & 0 \\
f''(\tilde u(t_n+s)) \tilde v(t_n+s) & f'(\tilde u(t_n+s))
\end{array} \right)
\bigg[ \left(\begin{array}{cc}
0 &  1 \\
\Delta & 0
\end{array} \right) \left(\begin{array}{cc}
\tilde u(t_n+s)\\
\tilde v(t_n+s)
\end{array} \right) \bigg] \notag\\
=&e^{s L} \left(\begin{array}{cc}
0\\
 f''(\tilde u(t_n+s))(|\tilde v(t_n+s)|^2-|\nabla\tilde u(t_n+s)|^2)
\end{array} \right)  .
\end{align}
Note that the second-order partial derivatives are cancelled in \eqref{cancellation}. This cancellation structure in the semlinear Klein--Gordon equation has not been discovered before. It allows us to compute $I_1(t_n)$ without requiring the second-order partial derivatives and therefore improves the accuracy of the numerical approximation for low-regularity solutions. As a result, the remainder can be estimated as follows:
\begin{align}
\| R_*(t_n) \|_1
&\lesssim \tau^3 \max_{s\in[0,\tau]} (\|\nabla\tilde u(t_n+s)\|_{L^{4}}^2 + \|\tilde v(t_n+s)\|_{L^4}^2) \notag\\
&\lesssim \tau^3 \max_{s\in[0,\tau]} (\|\tilde u(t_n+s)\|_{H^{1+\frac{d}{4}}}^2 + \|\tilde v(t_n+s)\|_{H^{\frac{d}{4}}}^2) \notag\\
&\lesssim \tau^3  \| U(t_n)\|_{1+\frac{d}{4}}^2 .
\end{align}
In the case $d=1$, the following result holds:
\begin{align}
\| R_*(t_n) \|_{\frac{1}{2}-\epsilon}
&\lesssim \tau^3 \max_{s\in[0,\tau]} (\| |\nabla\tilde u(t_n+s)|^2 \|_{H^{-\frac{1}{2}-\epsilon}}^2 + \|\tilde v(t_n+s)^2\|_{H^{-\frac{1}{2}-\epsilon}}) \notag\\
&\lesssim \tau^3 \max_{s\in[0,\tau]} (\| \nabla\tilde u(t_n+s) \|_{L^2}^2 + \|\tilde v(t_n+s)\|_{L^2}^2) \notag\\
&\lesssim \tau^3  \| U(t_n)\|_{1}^2 .
\end{align}
By the definition of $R_*(t_n)$ in \eqref{I1-expr10} and the triangle inequality, we also obtain
\begin{align}
\| R_*(t_n) \|_{2}
&\lesssim
\bigg\|\int_0^\tau e^{\tau L}G(t_n+s)\d s\bigg\|_2 \notag\\
&\quad+
\bigg\|\tau e^{\tau L}F(U(t_n))
+ (2L)^{-1} \big[ \tau e^{\tau L} - (2L)^{-1}(e^{\tau L} - e^{-\tau L})\big]
\left(\begin{array}{cc}
-f(u(t_n)) \\
f'(u(t_n)) \partial_tu(t_n)
\end{array} \right) \bigg\|_2 \notag\\
&\lesssim \tau \| U(t_n)\|_{1} .
\end{align}
Therefore, the Sobolev interpolation inequality implies that
\begin{align}\label{R-star-d=1}
\| R_*(t_n) \|_{1} \lesssim
\| R_*(t_n) \|_{\frac{1}{2}-\epsilon}^{\frac{1}{3/2+\epsilon}}\| R_*(t_n) \|_{2}^{\frac{1/2+\epsilon}{3/2+\epsilon}}
\lesssim \tau^{\frac{7/2+\epsilon}{3/2+\epsilon}} ( \| U(t_n)\|_{1} + \| U(t_n)\|_{1}^2) .
\end{align}

 \medskip

{\it Part 2: Approximation to $I_2(t)$.}

By approximating $U(t_n+\sigma)$ with $e^{\sigma L}U(t_n)$ in the expression of $I_2(t_n)$ in \eqref{2nd-expr}, 
we have
\begin{align}\label{I2-expr1}
I_2(t_n)
=&\,
 \int_0^\tau  e^{(\tau-s)L} \bigg[ F'(e^{sL}U(t_n))\int_0^s e^{(s-\sigma)L}F(U(t_n+\sigma))\d \sigma\bigg]  \d s  \notag\\
=&\,
 \int_0^\tau  e^{(\tau-s)L} \bigg[ F'(e^{sL}U(t_n))\int_0^s e^{(s-\sigma)L}F(e^{\sigma L}U(t_n))\d \sigma\bigg]  \d s
+R_{41}(t_n)  \notag\\
=&\,
 \int_0^\tau   e^{(\tau-s)L} \bigg[ F'(e^{sL}U(t_n)) s e^{sL}F(U(t_n))\bigg]  \d s
+R_{41}(t_n)  +R_{42}(t_n)  \notag\\
=&\,
 \int_0^\tau  s e^{\tau L} \big[ F'(U(t_n)) F(U(t_n))\big]  \d s
+ R_{41}(t_n)  +R_{42}(t_n)  +R_{43}(t_n)   \notag\\
=&\,
R_{41}(t_n)  +R_{42}(t_n)  +R_{43}(t_n) ,
\end{align}
where the last equality uses the property
$$
F'(U(t_n)) F(U(t_n)) = \left(\begin{array}{cc}
0 &  0 \\
f'(u) & 0
\end{array} \right)
\left(\begin{array}{cc}
0  \\
f(u)
\end{array} \right)
= \left(\begin{array}{cc}
0  \\
0
\end{array} \right) ,
$$
and the remainders $R_{4j}(t_n)$, $j=1,2,3$, are defined by
\begin{align}
R_{41}(t_n)
=&
 \int_0^\tau  e^{(\tau-s)L} \bigg[ F'(e^{sL}U(t_n))\int_0^s e^{(s-\sigma)L}
[F(U(t_n+\sigma))-F(e^{\sigma L}U(t_n))]\d \sigma\bigg]  \d s , \\
R_{42}(t_n)
=&
 \int_0^\tau  e^{(\tau-s)L} \bigg[ F'(e^{sL}U(t_n))\int_0^s
\Big( e^{(s-\sigma)L}F(e^{\sigma L}U(t_n)) - e^{sL}F(U(t_n)) \Big) \d \sigma\bigg]  \d s, \\
R_{43}(t_n)
=&
 \int_0^\tau  s e^{\tau L} \Big( e^{-sL} \big[ F'(e^{sL}U(t_n)) e^{sL}F(U(t_n))\big] - F'(U(t_n)) F(U(t_n)) \Big) \d s  .
\end{align}
The three remainders $R_{41}(t_n)  $, $R_{42}(t_n)  $ and $R_{43}(t_n)  $ are estimated as follows.

Firstly,
\begin{align}
\| R_{41}(t_n)  \|_{1}
\lesssim&\,
 \int_0^\tau  \bigg\| F'(e^{sL}U(t_n))\int_0^s e^{(s-\sigma)L}
[F(U(t_n+\sigma))-F(e^{\sigma L}U(t_n))]\d \sigma\bigg]  \bigg\|_1 \d s \notag\\
\lesssim&\,
 \int_0^\tau  \bigg\| \int_0^s e^{(s-\sigma)L}
[F(U(t_n+\sigma))-F(e^{\sigma L}U(t_n))]\d \sigma\bigg]  \bigg\|_0 \d s \notag\\
\lesssim&\,
\tau^2
\max_{\sigma\in[0,\tau]} \| U(t_n+\sigma))-e^{\sigma L}U(t_n))\|_0
\end{align}
By using \eqref{Duhamel2} we obtain that
\begin{align}
\| U(t_n+s))-e^{s L}U(t_n))\|_0
\lesssim
s\max_{\sigma\in[0,s]} \| U(t_n+\sigma) \|_0 .
\end{align}
Then, substituting this result into the estimate of $\| R_{41}(t_n) \|_{1}$, we obtain
\begin{align}
\| R_{41}(t_n)  \|_{1}
\lesssim&\,
\tau^3 \max_{t\in[t_n,t_{n+1}]} \| U(t) \|_0
\end{align}

Secondly, substituting the identity $$e^{(s-\sigma)L}F(e^{\sigma L}U(t_n)) - e^{sL}F(U(t_n)) = e^{s L} \int_0^\sigma \frac{\d}{\d\rho} e^{-\rho L}F(e^{\rho L}U(t_n)) \d\rho$$ into the expression of $R_{42}(t_n)$ yields
\begin{align}
R_{42}(t_n)
=&
 \int_0^\tau  e^{(\tau-s)L} \bigg[ F'(e^{sL}U(t_n))\int_0^s e^{s L}
\int_0^\sigma \frac{\d}{\d\rho} e^{-\rho L}F(e^{\rho L}U(t_n)) \d\rho \d \sigma\bigg]  \d s .
\end{align}
From this expression we immediately obtain
\begin{align}
\| R_{42}(t_n) \|_1
\lesssim&\,
 \int_0^\tau  \bigg\| F'(e^{sL}U(t_n))\int_0^s e^{s L}
\int_0^\sigma \frac{\d}{\d\rho} e^{-\rho L}F(e^{\rho L}U(t_n)) \d\rho \d \sigma\bigg\|_1  \d s \notag\\
\lesssim&\,
 \int_0^\tau  \bigg\| \int_0^s e^{s L}
\int_0^\sigma \frac{\d}{\d\rho} e^{-\rho L}F(e^{\rho L}U(t_n)) \d\rho \d \sigma\bigg\|_0  \d s \notag\\
\lesssim&\,
 \int_0^\tau  \int_0^s
\int_0^\sigma \bigg\| \frac{\d}{\d\rho} e^{-\rho L}F(e^{\rho L}U(t_n))  \bigg\|_0 \d\rho \d \sigma  \d s \notag\\
\lesssim&\,
\tau^3 \| U(t_n)\|_1 ,
\end{align}
where we have used \eqref{R2-estimate2} in the last inequality.

Thirdly, we have
\begin{align}\label{R43-1}
\| R_{43}(t_n) \|_1
=&\,
\bigg\| \int_0^\tau  s e^{\tau L} \int_0^s \frac{\d}{\d\sigma} e^{-\sigma L} \big[ F'(e^{\sigma L}U(t_n)) e^{\sigma L}F(U(t_n))\big] \d\sigma \d s \bigg\|_1 \notag\\
\lesssim&\,
 \int_0^\tau  s \int_0^s \bigg\| \frac{\d}{\d\sigma} e^{-\sigma L} \big[ F'(e^{\sigma L}U(t_n)) e^{\sigma L}F(U(t_n))\big] \bigg\|_1   \d\sigma \d s \notag\\
\lesssim&\,
\tau^3 \max_{\sigma\in[0,\tau]} \bigg\| \frac{\d}{\d\sigma} e^{-\sigma L} \big[ F'(e^{\sigma L}U(t_n)) e^{\sigma L}F(U(t_n))\big] \bigg\|_1  .
\end{align}
Let
$$
\left(\begin{array}{cc}
\tilde p(t_n+\sigma) \\
\tilde q(t_n+\sigma)
\end{array} \right)
=e^{\sigma L}F(U(t_n))
=e^{\sigma L}
\left(\begin{array}{cc}
0 \\
f(u(t_n))
\end{array} \right) .
$$
which satisfies the following estimate according to the basic estimates in \eqref{Hk-estimates}:
\begin{align}\label{Hk-pq}
\|\tilde p(t_n+\sigma)\|_{H^k(\Omega)}
+\|\tilde q(t_n+\sigma)\|_{H^{k-1}(\Omega)}
\lesssim
\|f(u(t_n))\|_{H^{k-1}(\Omega)}\quad\mbox{for}\,\,\, k=1,2.
\end{align}
Then
\begin{align}
&\frac{\d}{\d\sigma}  F'(e^{\sigma L}U(t_n)) e^{\sigma L}F(U(t_n)) \notag\\
&=\frac{\d}{\d\sigma}
\left[
\left(\begin{array}{cc}
0 & 0 \\
f'(\tilde u(t_n+\sigma)) & 0
\end{array} \right)
\left(\begin{array}{cc}
\tilde p(t_n+\sigma) \\
\tilde q(t_n+\sigma)
\end{array} \right) \right] \notag\\
&=\frac{\d}{\d\sigma}
\left(\begin{array}{cc}
0 \\
f'(\tilde u(t_n+\sigma))\tilde p(t_n+\sigma)
\end{array} \right) \notag\\
&=\left(\begin{array}{cc}
0 & 0 \\
f''(\tilde u(t_n+\sigma))\tilde p(t_n+\sigma)  & 0
\end{array} \right)
\bigg[\left(\begin{array}{cc}
0 &  1 \\
\Delta & 0
\end{array} \right)
\left(\begin{array}{cc}
\tilde u(t_n+\sigma) \\
\tilde v(t_n+\sigma)
\end{array} \right) \bigg] \notag\\
&\quad
+
\left(\begin{array}{cc}
0 & 0 \\
f'(\tilde u(t_n+\sigma)) & 0
\end{array} \right)
\bigg[\left(\begin{array}{cc}
0 &  1 \\
\Delta & 0
\end{array} \right)
\left(\begin{array}{cc}
\tilde p(t_n+\sigma) \\
\tilde q(t_n+\sigma)
\end{array} \right) \bigg] \notag\\
&=\left(\begin{array}{cc}
0 \\
f''(\tilde u(t_n+\sigma))\tilde p(t_n+\sigma)\tilde v(t_n+\sigma)
+ f'(\tilde u(t_n+\sigma))\tilde q(t_n+\sigma)
\end{array} \right) ,
\end{align}
and therefore
\begin{align}
&\frac{\d}{\d\sigma} e^{-\sigma L} \big[ F'(e^{\sigma L}U(t_n)) e^{\sigma L}F(U(t_n))\big] \notag\\
&=-e^{-\sigma L}LF'(e^{\sigma L}U(t_n)) e^{\sigma L}F(U(t_n))
+ e^{-\sigma L} \frac{\d}{\d\sigma}  F'(e^{\sigma L}U(t_n)) e^{\sigma L}F(U(t_n)) \notag\\
&=
e^{-\sigma L}
\left(\begin{array}{cc}
- f'(\tilde u(t_n+\sigma))\tilde p(t_n+\sigma) \\
f''(\tilde u(t_n+\sigma))\tilde p(t_n+\sigma)\tilde v(t_n+\sigma)
+ f'(\tilde u(t_n+\sigma))\tilde q(t_n+\sigma)
\end{array} \right) .
\end{align}
This implies that
\begin{align}
&
\bigg\| \frac{\d}{\d\sigma} e^{-\sigma L} \big[ F'(e^{\sigma L}U(t_n)) e^{\sigma L}F(U(t_n))\big] \bigg\|_1 \notag\\
&\lesssim
\|f'(\tilde u(t_n+\sigma))\tilde p(t_n+\sigma)\|_{H^1(\Omega)} \notag\\
&\quad\, +
\big\| f''(\tilde u(t_n+\sigma))\tilde p(t_n+\sigma)\tilde v(t_n+\sigma)
+ f'(\tilde u(t_n+\sigma))\tilde q(t_n+\sigma) \big\|_{L^2(\Omega)} \notag\\
&\lesssim
\| \tilde p(t_n+\sigma)\|_{H^1(\Omega)}
+\| \tilde p(t_n+\sigma)\|_{L^\infty(\Omega)} \|\tilde u(t_n+\sigma) \|_{H^1(\Omega)} \notag\\
&\quad\,
+\| \tilde p(t_n+\sigma)\|_{L^\infty(\Omega)} \|\tilde v(t_n+\sigma) \|_{L^2(\Omega)}
+ \|\tilde q(t_n+\sigma) \|_{L^2(\Omega)} \notag\\
&\lesssim
\| \tilde p(t_n+\sigma)\|_{H^{\frac32+\epsilon}(\Omega)}
(\|\tilde u(t_n+\sigma) \|_{H^1(\Omega)}+\|\tilde v(t_n+\sigma) \|_{L^2(\Omega)}) \notag\\
&\quad\,
+\| \tilde p(t_n+\sigma)\|_{H^1(\Omega)}  + \|\tilde q(t_n+\sigma) \|_{L^2(\Omega)} \notag\\
&\lesssim
\| f(u(t_n)) \|_{H^{\frac12+\epsilon}(\Omega)} (\|\tilde u(t_n+\sigma) \|_{H^1(\Omega)}+\|\tilde v(t_n+\sigma) \|_{L^2(\Omega)}) \notag\\
&\quad\,
+\| \tilde p(t_n+\sigma)\|_{H^1(\Omega)}  + \|\tilde q(t_n+\sigma) \|_{L^2(\Omega)} \notag\\
&\lesssim
\|U(t_n)\|_{1}+\|U(t_n)\|_{1}^2 + \|U(t_n)\|_{0} .
\end{align}
By substituting this result into \eqref{R43-1}, we obtain
\begin{align}\label{R43}
\| R_{43}(t_n) \|_1
\lesssim&\,
\tau^3 (\|U(t_n)\|_{1}+\|U(t_n)\|_{1}^2 + \|U(t_n)\|_{0}) .
\end{align}
Therefore, from \eqref{I2-expr1} we obtain
\begin{align}\label{R4}
\| I_2(t_n) \|_1
&\lesssim \| R_{41}(t_n) \|_1 + \| R_{42}(t_n) \|_1 + \| R_{43}(t_n) \|_1 \notag\\
&\lesssim
\tau^3 (\|U(t_n)\|_{1}+\|U(t_n)\|_{1}^2 + \|U(t_n)\|_{0}) .
\end{align}

Therefore,  substituting expressions \eqref{I1-expr10} and \eqref{I2-expr1} into \eqref{2nd-expr} yields
\begin{align}\label{2nd-expr-2}
U(t_{n+1})
=&\,
e^{\tau L}U(t_n) + \tau e^{\tau L}F(U(t_n))
+ (2L)^{-1} \big[ \tau e^{\tau L} - (2L)^{-1}(e^{\tau L} - e^{-\tau L})\big]
H(U(t_n)) \notag\\
&\,
+ I_2(t_n) + R_*(t_n) + R_3(t_n) ,
\end{align}
where
$$
H(U(t_n)):= \left(\begin{array}{cc}
-f(u(t_n)) \\
f'(u(t_n)) \partial_tu(t_n)
\end{array} \right) .
$$
By dropping the remainders $R_*(t_n) $ and $R_3(t_n) $, we obtain the following numerical method:
\begin{align}\label{2nd-method0}
U^{n+1}
=&\,
e^{\tau L}U^{n} + \tau e^{\tau L}F(U^{n})
+ (2L)^{-1} \big[ \tau e^{\tau L} - (2L)^{-1}(e^{\tau L} - e^{-\tau L})\big]
H(U^n) ,
\end{align}
which can also be written as \eqref{2nd-method}.

In view of \eqref{2nd-method}, the new method we constructed here turns out to be a correction of the Lie splitting method without requiring second-order partial derivatives of the solution, i.e., it improves the accuracy of the Lie splitting method under low-regularity conditions.

\section{The spatial discretization}\label{section:spatial}

Let $\Omega=[0,1]^d$.
It is known that any function $V\in H^1_0(\Omega)\times L^2(\Omega)$ can be expanded into the Fourier sine series, i.e.,
\begin{align}
V = \sum_{n_1,\cdots,n_d=1}^\infty V_{n_1,\cdots,n_d} \sin(n_1\pi x_1)\cdots \sin(n_d\pi  x_d) .
\end{align}
Let
$$
S_N= \bigg\{ \sum_{n_1,\cdots,n_d=1}^N V_{n_1,\cdots,n_d} \sin(n_1\pi x_1)\cdots \sin(n_d\pi  x_d) :
V_{n_1,\cdots,n_d}\in \R^2  \bigg\} ,
$$
and denote by $I_N$ the trigonometric interpolation operator onto $S_N$. We consider the following fully discrete spectral method for the second-order low-regularity integrator in \eqref{2nd-method0}:
\begin{align}\label{space-2}
U^{n+1}_N
=&\,
e^{\tau L}U^n_N
+ \frac{\tau}{2} e^{\tau L} I_N F(U^n_N) \notag\\
&\,
+  (2L)^{-1} \big[ \tau e^{\tau L} - (2L)^{-1}(e^{\tau L} - e^{-\tau L})\big]
I_N H(U^n_N) .
\end{align}
For given $U^n_N$, the trigonometric interpolations $I_N F(U^n_N)$ and $I_N H(U^n_N)$ can be computed with FFT.

Let $E_N^n=\Pi_N U(t_n)-U^n_N$ be the error of the numerical solution.
Since the exact solution satisfies
\begin{align}\label{2nd-expr-3}
\Pi_NU(t_{n+1})
=&\,
e^{\tau L}\Pi_NU(t_{n}) + \frac{\tau }{2} e^{\tau L} \Pi_N F(U(t_{n})) \notag\\
&\,
+  (2L)^{-1} \big[ \tau e^{\tau L} - (2L)^{-1}(e^{\tau L} - e^{-\tau L})\big]
\Pi_N H(U(t_{n}) ) \notag\\
&\,
+ \Pi_N [I_2(t_n)+R_*(t_n)+R_3(t_n)] ,
\end{align}
the difference between \eqref{2nd-expr-3} and \eqref{space-2} yields the following error equation:
\begin{align}\label{error-2}
E^{n+1}_N
=&\,
e^{\tau L}E^n_N
+ \frac{\tau}{2} e^{\tau L} \Pi_N ( F(U(t_{n})) - F(U^n_N) ) \notag\\
&\, + (2L)^{-1} \big[ \tau e^{\tau L} - (2L)^{-1}(e^{\tau L} - e^{-\tau L})\big]
\Pi_N \big( H(U(t_{n}) - H(U^n_N) \big)  \notag\\
&\, + \Pi_N I_2(t_n) + \Pi_N R_*(t_n)  + \Pi_N R_3(t_n) + R_5(t_n) + R_6(t_n) ,
\end{align}
with
\begin{align}
R_5(t_n)
=&\,
\frac{\tau}{2} e^{\tau L} ( \Pi_N  - I_N ) F(U^n_N) , \notag\\
R_6(t_n)
=&\,
(2L)^{-1} \big[ \tau e^{\tau L} - (2L)^{-1}(e^{\tau L} - e^{-\tau L})\big]
( \Pi_N  - I_N) H(U^n_N) .
\end{align}

The following result shows that for a solution bounded in the energy space $H^1_0(\Omega)\times L^2(\Omega)$ the proposed numerical method can have second-order convergence in time and first-order convergence in space in the same energy space.

\begin{theorem}\label{THM:2nd-H1}
For $d=1,2,3$ and $U\in L^\infty(0,T;H^{1+\frac{d}{4}}(\Omega)\cap H^1_0(\Omega)\times H^{\frac{d}{4}}(\Omega))$, the numerical solution given by \eqref{space-2}, with initial value $U^0_N=\Pi_NU(t_n)$, has the following error bound:
\begin{align}
\max_{0\le n\le T/\tau} \|E_N^n\|_1 \lesssim
&\,  \tau^2 +N^{-1-\frac{d}{4}} .
\label{H1-error}
\end{align}
\end{theorem}
\begin{proof}
If $U\in L^\infty(0,T;H^{1+\frac{d}{4}}(\Omega)\cap H^1_0(\Omega)\times H^{\frac{d}{4}}(\Omega))$ then \eqref{tR3} and \eqref{R4} imply that the remainders $\Pi_N I_2(t_n)$ and $\Pi_N R_3(t_n)$ satisfy the following estimates:
\begin{align}\label{R3R4}
\|\Pi_N I_2(t_n)\|_1 + \|\Pi_N R_*(t_n)\|_1 + \|\Pi_N R_3(t_n)\|_1
\lesssim &\, \tau^3 \quad\mbox{in the case $d=1,2,3$}.
\end{align}
The remainders $R_5(t_n)$ and $R_6(t_n)$ can be estimated by using mathematical induction on $n$: assuming that
\begin{align}\label{induction-H1}
\|U^n_N\|_1 \le \|\Pi_N U(t_n)\|_1 + 1
\end{align}
we shall prove the following results:
\begin{align}\label{induction-H1-2}
\|U^{n+1}_N\|_1 \le \|\Pi_N U(t_{n+1})\|_1 + 1
\quad\mbox{and}\quad
\|E_N^n\|_1 \lesssim \tau+N^{-1} .
\end{align}
Under assumption \eqref{induction-H1} we have
\begin{align*} 
\| R_5(t_n) \|_1
\lesssim&\, \tau \| ( \Pi_N  - I_N ) F(U^n_N)  \|_1 \notag\\
\lesssim&\, \tau \| ( \Pi_N  - I_N ) f(u^n_N)  \|_{L^2} \notag\\
\lesssim&\, \tau N^{-2} \|f(u^n_N)\|_{H^2} \notag\\
\lesssim&\, \tau N^{-2} \|f'(u^n_N)\nabla^2 u^n_N + f''(u^n_N)\nabla u^n_N\otimes \nabla u^n_N  \|_{L^2} \notag\\
\lesssim&\, \tau N^{-2} ( \|u^n_N\|_{H^2} + \|\nabla u^n_N\|_{L^4}^2 ) \\[5pt]
\| R_6(t_n) \|_1 
=&\,  \big\| (2L)^{-1}\big[ \tau e^{\tau L} - (2L)^{-1}(e^{\tau L} - e^{-\tau L})\big]
( \Pi_N  - I_N) H(U^n_N)\big\|_1 \notag\\
\lesssim&\, \big\|\big[ \tau e^{\tau L} - (2L)^{-1}(e^{\tau L} - e^{-\tau L})\big]
( \Pi_N  - I_N) H(U^n_N)\big\|_0 \notag\\
\lesssim&\, \tau \big\| ( \Pi_N  - I_N) H(U^n_N)\big\|_0 \notag\\
\lesssim&\, \tau N^{-2}
(\|f(u^n_N)\|_{H^2}+ \| f'(u^n_N) v^n_N\|_{H^1}) \notag\\
\lesssim&\, \tau N^{-2}
\|f'(u^n_N)\nabla^2 u^n_N + f''(u^n_N)\nabla u^n_N\otimes \nabla u^n_N  \|_{L^2} ) \notag\\
&\,
+ \tau N^{-2} (\| f''(u^n_N)  v^n_N\nabla u^n_N + f'(u^n_N)\nabla v^n_N \|_{L^2} ) \notag\\
\lesssim&\, \tau N^{-2}
(\| u^{n}_N \|_{H^2} + \| \nabla u^{n}_N \|_{L^4}^2)
+ \tau N^{-2} ( \| v^{n}_N \|_{L^4}\| \nabla u^{n}_N \|_{L^4}
+\| \nabla v^{n}_N \|_{L^2} ) .
\end{align*}
In the case $d=1,2,3$, the Sobolev interpolation inequality
$ \|\nabla u^n_N\|_{L^4}\le  \|  u^{n}_N \|_{H^{1+\frac{d}{4}}}$ implies that
\begin{align} \label{R6}
\| R_5(t_n) \|_1
\lesssim&\, \tau N^{-2} ( \|u^n_N\|_{H^2} +  \|  u^{n}_N \|_{H^{1+\frac{d}{4}}}^2 ) \notag\\
\lesssim&\, \tau N^{-1-\frac{d}{4}} ( \|u^n_N\|_{H^{1+\frac{d}{4}}} + \| u^n_N\|_{H^1}^2 ) , \\[5pt]
\| R_6(t_n) \|_1 \label{R7}
\lesssim&\, \tau N^{-2}
(\| u^{n}_N \|_{H^2} + \| v^{n}_N \|_{H^1}  + \|  u^{n}_N \|_{H^{1+\frac{d}{4}}}^2 + \|  v^{n}_N \|_{H^{\frac{d}{4}}}^2) \notag\\
\lesssim&\, \tau N^{-1-\frac{d}{4}}
(\| u^{n}_N \|_{H^{1+\frac{d}{4}}} + \|  v^{n}_N \|_{H^{\frac{d}{4}}}
+ \|  u^{n}_N \|_{H^{1+\frac{d}{4}}}^2 + \|  v^{n}_N \|_{H^{\frac{d}{4}}}^2 ) .
\end{align}
%
By using these estimates and taking the energy norm $|\cdot|_1$ on both sides of \eqref{error-2}, we obtain
\begin{align}
| E^{n+1}_N |_1
\le (1 + C\tau) | E^n_N |_1 + C\tau (\tau^2+N^{-1-\frac{d}{4}})  .
\end{align}
Then, using Gronwall's inequality and the equivalence of norms $|\cdot|_1\sim \|\cdot\|_1$ on the energy space $H^1_0(\Omega)\times L^2(\Omega)$, we obtain the following error bound:
\begin{align}
\| E^{n+1}_N \|_1
\lesssim &\, \tau^2+N^{-1-\frac{d}{4}} .
\end{align}
There exist some positive constants $\tau_0$ and $N_0$ such that for $\tau\le\tau_0$ and $N\ge N_0$ we obtain
\begin{align}
\| E^{n+1}_N \|_1
\le 1 .
\end{align}
This proves \eqref{induction-H1-2} (with an additional triangle inequality).
\end{proof}


\begin{remark}
\upshape
By passing to the limit $N\rightarrow\infty$ in Theorem \ref{THM:2nd-H1},
one can obtain the semi-discretization results in Theorem \ref{THM:2nd-H1-TD}.
\end{remark}

\begin{remark}\label{remarkTHM}
\upshape
In the case $d=1$, the remainder $R_*(t_n)$ can be estimated by using \eqref{R-star-d=1}, which yields the following result:
\begin{align}
\max_{0\le n\le T/\tau} \|E_N^n\|_1 \lesssim
&\,  \tau^{\frac{5}{3}-\epsilon}  +N^{-1} \quad\mbox{(for any fixed $\epsilon>0$)} .
\label{H1-error-d=1}
\end{align}
This result holds under the weaker regularity condition
$U\in C([0,T];H^1_0(\Omega)\times L^2(\Omega))$, i.e.,
the numerical solution has higher-order convergence in the energy space
without requiring additional regularity in the solution.
\end{remark}

\begin{remark}
\upshape
For any given initial value $(u^0,v^0)\in H^{1+\frac{d}{4}}(\Omega)\cap H^1_0(\Omega)\times H^{\frac{d}{4}}(\Omega)$,
Theorem \ref{THM:2nd-H1} states that the error of the numerical solution is as follows:
$$
\|\Pi_N u(t_n)-u^n_N\|_{H^1}+\|\Pi_N v(t_n)-v^n_N\|_{L^2}\lesssim \tau^2+N^{-1-\frac{d}{4}} ,
$$
which is a superconvergence result that much better than the regularity of the solution
in both time and space.
In general, for any fixed $t$, the projection error in space satisfies
$$
\|\Pi_N u(t)-u(t)\|_{H^1}+\|\Pi_N v(t)-v(t)\|_{L^2} \lesssim N^{-\frac{d}{4}} .
$$
\end{remark}

\section{Numerical experiments}\label{sec:numExp}


In this section we present numerical experiments to support the theoretical analysis and to illustrate the performance of our new method in \eqref{2nd-method} on the semilinear Klein--Gordon equation \eqref{pde} in a one-dimensional domain $\Omega=[0,1]$ with $f(x) = \sin(x)$, using $N = 2^{12}$ terms of a Fourier space discretization.
As for the initial state of the differential equation, we generate, as described in Section 5.1 of \cite{OS18}, random initial data $u^0$ and $u_t^0$ from the space $H^{\theta}(\Omega)$ such that $\|u^0\|_{L^2}= 1$ and $\|u_t^0\|_{L^2}= 1$. In particular, we are interested in comparing the smooth case ($\theta \rightarrow \infty$) with the low-regularity case ($\theta = 1$).

Our new method is tested in comparison with several well-established numerical techniques for the semilinear Klein--Gordon equation.
To define them, it is useful to introduce the operator $\Sigma=\sqrt{-\Delta}$, which satisfies that $\Delta = -\Sigma^2$. This is because the exponential of our linear operator can be easily expressed as
\[
\exp(L)
=
\exp
\left(t
\begin{pmatrix}
       0 & 1 \\
\Delta &   0 \\
\end{pmatrix}
\right)
=
\exp
\left(t
\begin{pmatrix}
       0 & 1 \\
-\Sigma^2 &   0 \\
\end{pmatrix}
\right)
=
\begin{pmatrix}
       \cos( t \Sigma ) & t \sinc( t \Sigma ) \\
-\Sigma\sin( t \Sigma ) &    \cos( t \Sigma ) \\
\end{pmatrix}
.
\]
This expression is worth using only in case the operator $\Delta$ can be discretized in space by means of a diagonal matrix or if the resulting discretization matrix's size is particularly modest.
In fact, in the other cases, computing the matrix square root is generally unfeasible.
The above-mentioned numerical techniques are:\medskip
\begin{itemize}
\item The second-order low-regularity exponential-type scheme from \cite{RS21}, that we refer to as \texttt{rs21}.
This method computes approximations $u^{n+1}$, $v^{n+1}$ to $u(t_{n+1})$, $u_t(t_{n+1})$ at discrete times $t_{n+1}= t_0 +(n+1)\tau$ with the time step size $\tau$ as
\[
\begin{pmatrix}
u^{n+\frac12} \\
v^{n+\frac12}
\end{pmatrix}
=
\exp
\left(\tau
\begin{pmatrix}
0 & 1 \\
\Delta &   0 \\
\end{pmatrix}
\right)
\begin{pmatrix}
u^{n} \\
v^{n}
\end{pmatrix},
\]
\[
\begin{pmatrix}
u^{n+1} \\
v^{n+1}
\end{pmatrix}
=
\begin{pmatrix}
u^{n+\frac12} \\
v^{n+\frac12}
\end{pmatrix}
+ \frac{\tau}{2}
\left(
\exp
\left(\tau
\begin{pmatrix}
0 & 1 \\
\Delta &   0 \\
\end{pmatrix}
\right)
\begin{pmatrix}
0 \\
\sin(u^{n})
\end{pmatrix}
+
\begin{pmatrix}
0 \\
\sin(u^{n+\frac12})
\end{pmatrix}
\right).
\]\medskip
\item The recent second-order IMEX method for semilinear second-order wave equations from \cite{HL2021}, that we refer to as \texttt{hl21}.
This method computes approximations $u^{n+1}$, $v^{n+1}$ to $u(t_{n+1})$, $u_t(t_{n+1})$ at discrete times $t_{n+1}= t_0 +(n+1)\tau$ with the time step size $\tau$ as
\[
\begin{split}
v^{n+\frac12} &=
\left( 1 - \frac{\tau^2}{4}\Delta \right)^{-1}
\left( v^{n} + \frac{\tau}{2}\sin(u^n) + \frac{\tau}{2}\Delta u^{n}  \right), \\
u^{n+1}       &=  u^{n} + \tau v^{n+\frac12}, \\
v^{n+1}       &= 2v^{n+\frac12} -v^{n} + \frac{\tau}{2}\left( \sin(u^{n+1}) - \sin(u^n) \right). \\
\end{split}
\]\medskip
\item Another natural choice for measuring the performances of our scheme is the class of second-order trigonometric integrators expressly designed for the discretization in time of the spatially discrete nonlinear Klein--Gordon equation with periodic boundary conditions.
This class of trigonometric integrators computes approximations $u^{n+1}$, $v^{n+1}$ to $u(t_{n+1})$, $u_t(t_{n+1})$ at discrete times $t_{n+1}= t_0 +(n+1)\tau$ with the time stepsize $\tau$ as
\[
\begin{pmatrix}
u^{n+1} \\
v^{n+1}
\end{pmatrix}
=
\exp
\left(\tau
\begin{pmatrix}
0 & 1 \\
\Delta &   0 \\
\end{pmatrix}
\right)
\begin{pmatrix}
u^n \\
v^n
\end{pmatrix}
+
\frac{\tau}{2}
\begin{pmatrix}
\tau \Psi \sin(\Psi u^n) \\
\Psi_0 \sin(\Psi u^n)
+
\Psi_1 \sin(\Psi u^n)
\end{pmatrix}.
\]
The matrices $\Phi, \Psi, \Psi_0$, and $\Psi_1$ are filters defined by
\[
\Phi   = \phi  (\tau\Sigma),\quad
\Psi   = \psi  (\tau\Sigma),\quad
\Psi_0 = \psi_0(\tau\Sigma),\quad
\Psi_1 = \psi_1(\tau\Sigma)
\]
with filter functions $\phi, \psi, \psi_0$, and $\psi_1$ that satisfy $\phi(0) = \psi(0) = \psi_0(0) = \psi_1(0) = 1$. The choice of such filters uniquely characterizes a method.
For even filter functions, the method is symmetric if and only if
\begin{equation}
\label{eq:lgpsiphi}
\psi(x) = \sinc(x) \psi_1(x), \quad \psi_0(x) = \cos(x) \psi_1(x),
\end{equation}
and it is symplectic if and only if
\[
\psi(x) = \sinc(x)\phi(x).
\]
Popular choices of the filter functions are
\begin{itemize}
\item[$(B)$] The one with $\psi(x) = \sinc  (x), \phi(x) = 1       $, $\psi_0$ and $\psi_1$ as in \eqref{eq:lgpsiphi}. This is the impulse method by Deuflhard \cite{D79}. 
\item[$(C)$] The one with $\psi(x) = \sinc^2(x), \phi(x) = \sinc(x)$, $\psi_0$ and $\psi_1$ as in \eqref{eq:lgpsiphi}. This is the mollified impulse method by Garc\'ia-Archilla, Sanz-Serna \& Skeel \cite{GASSS98}. 
\item[$(E)$] The one with $\psi(x) = \sinc^2(x), \phi(x) = 1       $, $\psi_0$ and $\psi_1$ as in \eqref{eq:lgpsiphi}. This is the trigonometric exponential-type integrator by Hairer \& Lubich\cite{HL00}. 
\item[$(G)$] The one with $\psi(x) = \sinc^3(x), \phi(x) = \sinc(x)$, $\psi_0$ and $\psi_1$ as in \eqref{eq:lgpsiphi}. This is the trigonometric exponential-type integrator by Grimm \& Hochbruck \cite{GH06}. 
\item[$(\tilde{B})$] The one with $\psi(x) = \chi_{[-\pi,\pi]}(x)\sinc(x), \phi(x) = \chi_{[-\pi,\pi]}(x)$ $\psi_0$ and $\psi_1$ as in \eqref{eq:lgpsiphi}. This is the method introduced by Gauckler \cite{G15}. 
\end{itemize}
For a precise overview and for more information on this class of trigonometric methods we refer the reader to~\cite{G15}.
In our tests it turned out that the methods $B$ and $\tilde B$ are neatly superior to all the other options, therefore we will only include these two into the data presentation, referring to them as, respectively, \texttt{d79} and \texttt{g15}.\medskip
\item The second order classical Strang splitting scheme from \cite{S1968}, that we refer to as \texttt{ss68}.
This method computes approximations $u^{n+1}$, $v^{n+1}$ to $u(t_{n+1})$, $u_t(t_{n+1})$ at discrete times $t_{n+1}= t_0 +(n+1)\tau$ with the time step size $\tau$ as
\[
\begin{pmatrix}
u^{n+\frac12} \\
v^{n+\frac12}
\end{pmatrix}
=
\exp
\left(\frac{\tau}{2}
\begin{pmatrix}
0 & 1 \\
\Delta &   0 \\
\end{pmatrix}
\right)
\begin{pmatrix}
u^{n} \\
v^{n}
\end{pmatrix},
\]
\[
\begin{pmatrix}
u^{n+1} \\
v^{n+1}
\end{pmatrix}
=
\exp
\left(\frac{\tau}{2}
\begin{pmatrix}
0 & 1 \\
\Delta &   0 \\
\end{pmatrix}
\right)
\begin{pmatrix}
u^{n+\frac12} \\
v^{n+\frac12} + \tau \sin(u^{n+\frac12})
\end{pmatrix}.
\]
\end{itemize}

Finally, we did not consider any method applied in combination with the rewriting suggested in formula \eqref{formula:rootrewriting}, in fact our numerical experience tells that in this case such procedure would not prove to be efficient.

The $H^1(\Omega)\times L^2(\Omega)$ errors of the numerical solutions given by the above-mentioned methods and our new method, the corrected Lie method (which we refer to as {\tt c\_lie}), are presented Figure \ref{fig:sineKG_th16} and Figure \ref{fig:sineKG_th1} for smooth initial data and nonsmooth $H^1(\Omega)\times L^2(\Omega)$ initial data, respectively. The numerical results in Figure \ref{fig:sineKG_th16} indicate that all methods have second-order convergence for sufficiently smooth initial data. The numerical results in Figure \ref{fig:sineKG_th1} shows that the new method proposed in this article has second-order convergence for the nonsmooth $H^1(\Omega)\times L^2(\Omega)$ initial data, while all other second-order methods are practically first-order convergent in this nonsmooth case.

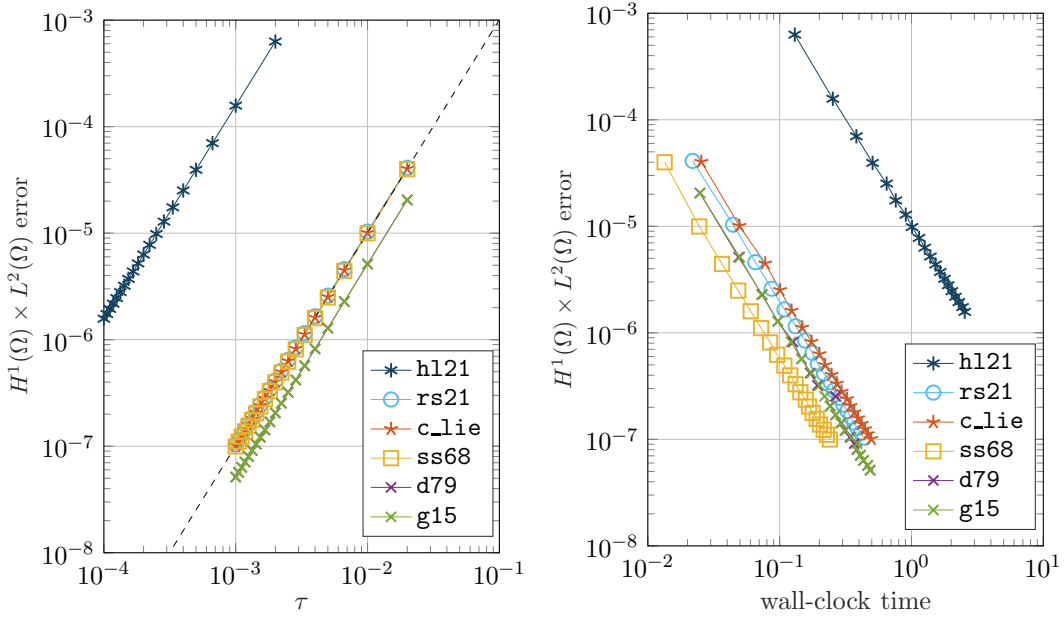
\begin{figure}[htp]
\scalebox{.9}{
%
%
\definecolor{color_rs21}{rgb}{0.30100,0.74500,0.93300}%
\definecolor{color_lsz}{rgb}{0.85000,0.32500,0.09800}%
\definecolor{color_ss68}{rgb}{0.92900,0.69400,0.12500}%
\definecolor{color_d79}{rgb}{0.49400,0.18400,0.55600}%
\definecolor{color_g15}{rgb}{0.46600,0.67400,0.18800}%
\definecolor{color_hl21}{rgb}{0.08240,0.26270,0.37650}%
\begin{tikzpicture}

\begin{axis}[%
width=2.300in,
height=3.100in,
at={(1.011in,0.642in)},
scale only axis,
xmode=log,
xmin=0.0001,
xmax=0.1,
xminorticks=true,
xlabel style={font=\color{white!15!black}},
xlabel={$\tau$},
ymode=log,
ymin=1.0e-08,
ymax=0.001,
yminorticks=true,
ylabel style={font=\color{white!15!black}},
ylabel={{\small $H^1(\Omega) \times L^2(\Omega)$ error}},
axis background/.style={fill=white},
xmajorgrids,
ymajorgrids,
legend style={at={(0.985,0.03)}, anchor=south east, legend cell align=left, align=left, draw=white!15!black}
]

\addplot [color=color_hl21, mark=asterisk, mark options={solid, color_hl21,thick,scale=1.5}]
  table[row sep=crcr]{%
0.002	0.00062965522136628\\
0.001	0.000157416677343595\\
0.000666666666666667	6.99632162128989e-05\\
0.0005	3.93543660645576e-05\\
0.0004	2.51868167624773e-05\\
0.000333333333333333	1.74908576358628e-05\\
0.000285714285714286	1.28504348503568e-05\\
0.00025	9.83862105606754e-06\\
0.000222222222222222	7.77373103775931e-06\\
0.0002	6.29672719624913e-06\\
0.000181818181818182	5.20391121092504e-06\\
0.000166666666666667	4.37273496685835e-06\\
0.000153846153846154	3.72588446389118e-06\\
0.000142857142857143	3.212628297873e-06\\
0.000133333333333333	2.79855937699573e-06\\
0.000125	2.45967426920317e-06\\
0.000117647058823529	2.17881470509863e-06\\
0.000111111111111111	1.94345158793849e-06\\
0.000105263157894737	1.74426372644879e-06\\
0.0001	1.57420046339118e-06\\
};
\addlegendentry{\texttt{hl21}}

\addplot [color=color_rs21, mark=o, mark options={solid, color_rs21,thick,scale=1.5}]
  table[row sep=crcr]{%
0.02	4.11075806568728e-05\\
0.01	1.03522656246232e-05\\
0.00666666666666667	4.60733454601047e-06\\
0.005	2.59279288160165e-06\\
0.004	1.65970041052919e-06\\
0.00333333333333333	1.15267247557078e-06\\
0.00285714285714286	8.46898611802093e-07\\
0.0025	6.48420109250919e-07\\
0.00222222222222222	5.12335721521192e-07\\
0.002	4.14991580105026e-07\\
0.00181818181818182	3.42966107456722e-07\\
0.00166666666666667	2.88183914114039e-07\\
0.00153846153846154	2.45549979746621e-07\\
0.00142857142857143	2.1172099795214e-07\\
0.00133333333333333	1.84429330166853e-07\\
0.00125	1.62093065365657e-07\\
0.00117647058823529	1.43581231278487e-07\\
0.00111111111111111	1.28068139138161e-07\\
0.00105263157894737	1.14939391766447e-07\\
0.001	1.0373027468517e-07\\
};
\addlegendentry{\texttt{rs21}}

\addplot [color=color_lsz, mark=star, mark options={solid, color_lsz,thick,scale=1.5}]
  table[row sep=crcr]{%
0.02	4.02854793247987e-05\\
0.01	1.00875269957212e-05\\
0.00666666666666667	4.48481779777051e-06\\
0.005	2.52300104641547e-06\\
0.004	1.61480270349603e-06\\
0.00333333333333333	1.12141792139945e-06\\
0.00285714285714286	8.2390784462645e-07\\
0.0025	6.30806452069714e-07\\
0.00222222222222222	4.98414048485665e-07\\
0.002	4.03713192605958e-07\\
0.00181818181818182	3.33644512247111e-07\\
0.00166666666666667	2.80351121414411e-07\\
0.00153846153846154	2.38876090732305e-07\\
0.00142857142857143	2.05966754394187e-07\\
0.00133333333333333	1.79417062410213e-07\\
0.00125	1.5768806774443e-07\\
0.00117647058823529	1.39679513029359e-07\\
0.00111111111111111	1.24588161481378e-07\\
0.00105263157894737	1.11816331321636e-07\\
0.001	1.00911910837585e-07\\
};
\addlegendentry{\texttt{c\_lie}}

\addplot [color=color_ss68, mark=square, mark options={solid, color_ss68,thick,scale=1.5}]
  table[row sep=crcr]{%
0.02	3.99159618338559e-05\\
0.01	9.96756339679678e-06\\
0.00666666666666667	4.42907721804812e-06\\
0.005	2.49116088107134e-06\\
0.004	1.59427966971982e-06\\
0.00333333333333333	1.10711069399703e-06\\
0.00285714285714286	8.13371932827046e-07\\
0.0025	6.22727731832249e-07\\
0.00222222222222222	4.92024285334989e-07\\
0.002	3.98533690318616e-07\\
0.00181818181818182	3.2936163099151e-07\\
0.00166666666666667	2.76750876015278e-07\\
0.00153846153846154	2.35807466519001e-07\\
0.00142857142857143	2.03320232908167e-07\\
0.00133333333333333	1.77111227233146e-07\\
0.00125	1.55661145060819e-07\\
0.00117647058823529	1.37883852170521e-07\\
0.00111111111111111	1.22986326133302e-07\\
0.00105263157894737	1.10378556093746e-07\\
0.001	9.96142631912341e-08\\
};
\addlegendentry{\texttt{ss68}}

\addplot [color=color_d79, mark=x, mark options={solid, color_d79,thick,scale=1.5}]
  table[row sep=crcr]{%
0.02	2.05402279832761e-05\\
0.01	5.13085795811599e-06\\
0.00666666666666667	2.28004966282204e-06\\
0.005	1.28247162756678e-06\\
0.004	8.20771470526284e-07\\
0.00333333333333333	5.69980890831017e-07\\
0.00285714285714286	4.18765343212946e-07\\
0.0025	3.20621943063987e-07\\
0.00222222222222222	2.53335700236259e-07\\
0.002	2.05206623579729e-07\\
0.00181818181818182	1.69596647955301e-07\\
0.00166666666666667	1.42512437372732e-07\\
0.00153846153846154	1.21434612704836e-07\\
0.00142857142857143	1.04710074419824e-07\\
0.00133333333333333	9.12175920507468e-08\\
0.00125	8.01749937247511e-08\\
0.00117647058823529	7.10231437935932e-08\\
0.00111111111111111	6.33538273654467e-08\\
0.00105263157894737	5.68632762858819e-08\\
0.001	5.13217256068615e-08\\
};
\addlegendentry{\texttt{d79}}

\addplot [color=color_g15, mark=x, mark options={solid, color_g15,thick,scale=1.5}]
  table[row sep=crcr]{%
0.02	2.05402279832761e-05\\
0.01	5.13085795811599e-06\\
0.00666666666666667	2.28004966282204e-06\\
0.005	1.28247162756678e-06\\
0.004	8.20771470526284e-07\\
0.00333333333333333	5.69980890831017e-07\\
0.00285714285714286	4.18765343212946e-07\\
0.0025	3.20621943063987e-07\\
0.00222222222222222	2.53335700236259e-07\\
0.002	2.05206623579729e-07\\
0.00181818181818182	1.69596647955301e-07\\
0.00166666666666667	1.42512437372732e-07\\
0.00153846153846154	1.21434612704836e-07\\
0.00142857142857143	1.04710074419824e-07\\
0.00133333333333333	9.12175920507468e-08\\
0.00125	8.01749937247511e-08\\
0.00117647058823529	7.10231437935932e-08\\
0.00111111111111111	6.33538273654467e-08\\
0.00105263157894737	5.68632762858819e-08\\
0.001	5.13217256068615e-08\\
};
\addlegendentry{\texttt{g15}}

\addplot [color=black, dashed]
  table[row sep=crcr]{%
9.765625000000000e-06     9.620215411874008e-12 \\
1.024000000000000e+01     1.057753870706535e+01 \\
};

\end{axis}
\end{tikzpicture}%
%
%
\definecolor{color_rs21}{rgb}{0.30100,0.74500,0.93300}%
\definecolor{color_lsz}{rgb}{0.85000,0.32500,0.09800}%
\definecolor{color_ss68}{rgb}{0.92900,0.69400,0.12500}%
\definecolor{color_d79}{rgb}{0.49400,0.18400,0.55600}%
\definecolor{color_g15}{rgb}{0.46600,0.67400,0.18800}%
\definecolor{color_hl21}{rgb}{0.08240,0.26270,0.37650}%
\begin{tikzpicture}

\begin{axis}[%
width=2.300in,
height=3.100in,
at={(1.011in,0.642in)},
scale only axis,
xmode=log,
xmin=0.01,
xmax=10,
xminorticks=true,
xlabel style={font=\color{white!15!black}},
xlabel={wall-clock time},
ymode=log,
ymin=1.0e-08,
ymax=0.001,
yminorticks=true,
ylabel style={font=\color{white!15!black}},
ylabel={{\small $H^1(\Omega) \times L^2(\Omega)$ error}},
axis background/.style={fill=white},
xmajorgrids,
ymajorgrids,
legend style={at={(0.65,0.03)}, anchor=south west, legend cell align=left, align=left, draw=white!15!black}
]

\addplot [color=color_hl21, mark=asterisk, mark options={solid, color_hl21,thick,scale=1.5}]
  table[row sep=crcr]{%
0.130200355	0.00062965522136628\\
0.253077924	0.000157416677343595\\
0.381956928	6.99632162128989e-05\\
0.507022004	3.93543660645576e-05\\
0.648224572	2.51868167624773e-05\\
0.760130506	1.74908576358628e-05\\
0.905112248	1.28504348503568e-05\\
1.009722551	9.83862105606754e-06\\
1.141389611	7.77373103775931e-06\\
1.260759532	6.29672719624913e-06\\
1.391747632	5.20391121092504e-06\\
1.524026824	4.37273496685835e-06\\
1.652051764	3.72588446389118e-06\\
1.780289281	3.212628297873e-06\\
1.899374371	2.79855937699573e-06\\
2.0234	2.45967426920317e-06\\
2.153811882	2.17881470509863e-06\\
2.286770283	1.94345158793849e-06\\
2.44061979	1.74426372644879e-06\\
2.535613849	1.57420046339118e-06\\
};
\addlegendentry{\texttt{hl21}}

\addplot [color=color_rs21, mark=o, mark options={solid, color_rs21,thick,scale=1.5}]
  table[row sep=crcr]{%
0.021880986	4.11075806568728e-05\\
0.04437799	1.03522656246232e-05\\
0.065449898	4.60733454601047e-06\\
0.086864907	2.59279288160165e-06\\
0.108833812	1.65970041052919e-06\\
0.132445121	1.15267247557078e-06\\
0.156056645	8.46898611802093e-07\\
0.177350124	6.48420109250919e-07\\
0.194795145	5.12335721521192e-07\\
0.217141175	4.14991580105026e-07\\
0.244469042	3.42966107456722e-07\\
0.258730463	2.88183914114039e-07\\
0.283050065	2.45549979746621e-07\\
0.303984121	2.1172099795214e-07\\
0.326606723	1.84429330166853e-07\\
0.34696718	1.62093065365657e-07\\
0.367222822	1.43581231278487e-07\\
0.387456647	1.28068139138161e-07\\
0.412449978	1.14939391766447e-07\\
0.438735607	1.0373027468517e-07\\
};
\addlegendentry{\texttt{rs21}}

\addplot [color=color_lsz, mark=star, mark options={solid, color_lsz,thick,scale=1.5}]
  table[row sep=crcr]{%
0.025468903	4.02854793247987e-05\\
0.049533629	1.00875269957212e-05\\
0.077752529	4.48481779777051e-06\\
0.100596976	2.52300104641547e-06\\
0.123201334	1.61480270349603e-06\\
0.148156509	1.12141792139945e-06\\
0.174616967	8.2390784462645e-07\\
0.20045274	6.30806452069714e-07\\
0.221412763	4.98414048485665e-07\\
0.248285405	4.03713192605958e-07\\
0.271857113	3.33644512247111e-07\\
0.295659211	2.80351121414411e-07\\
0.324029348	2.38876090732305e-07\\
0.344897249	2.05966754394187e-07\\
0.368179529	1.79417062410213e-07\\
0.393577585	1.5768806774443e-07\\
0.418879642	1.39679513029359e-07\\
0.444675229	1.24588161481378e-07\\
0.465436505	1.11816331321636e-07\\
0.495036487	1.00911910837585e-07\\
};
\addlegendentry{\texttt{c\_lie}}

\addplot [color=color_ss68, mark=square, mark options={solid, color_ss68,thick,scale=1.5}]
  table[row sep=crcr]{%
0.013410757	3.99159618338559e-05\\
0.024560596	9.96756339679678e-06\\
0.036571217	4.42907721804812e-06\\
0.048499855	2.49116088107134e-06\\
0.060342604	1.59427966971982e-06\\
0.072306081	1.10711069399703e-06\\
0.08472447	8.13371932827046e-07\\
0.095716506	6.22727731832249e-07\\
0.108654693	4.92024285334989e-07\\
0.120287378	3.98533690318616e-07\\
0.132221259	3.2936163099151e-07\\
0.143585785	2.76750876015278e-07\\
0.157792699	2.35807466519001e-07\\
0.167166571	2.03320232908167e-07\\
0.178064621	1.77111227233146e-07\\
0.191028984	1.55661145060819e-07\\
0.202163646	1.37883852170521e-07\\
0.216038433	1.22986326133302e-07\\
0.228104353	1.10378556093746e-07\\
0.240136646	9.96142631912341e-08\\
};
\addlegendentry{\texttt{ss68}}

\addplot [color=color_d79, mark=x, mark options={solid, color_d79,thick,scale=1.5}]
  table[row sep=crcr]{%
0.024817393	2.05402279832761e-05\\
0.048880634	5.13085795811599e-06\\
0.073641982	2.28004966282204e-06\\
0.096888961	1.28247162756678e-06\\
0.126414647	8.20771470526284e-07\\
0.146618611	5.69980890831017e-07\\
0.170397959	4.18765343212946e-07\\
0.193744004	3.20621943063987e-07\\
0.265221462	2.53335700236259e-07\\
0.242492484	2.05206623579729e-07\\
0.265370419	1.69596647955301e-07\\
0.292488299	1.42512437372732e-07\\
0.315607211	1.21434612704836e-07\\
0.341553897	1.04710074419824e-07\\
0.368206998	9.12175920507468e-08\\
0.384728239	8.01749937247511e-08\\
0.409214059	7.10231437935932e-08\\
0.432739029	6.33538273654467e-08\\
0.464911373	5.68632762858819e-08\\
0.486563139	5.13217256068615e-08\\
};
\addlegendentry{\texttt{d79}}

\addplot [color=color_g15, mark=x, mark options={solid, color_g15,thick,scale=1.5}]
  table[row sep=crcr]{%
0.024868648	2.05402279832761e-05\\
0.049862783	5.13085795811599e-06\\
0.072942728	2.28004966282204e-06\\
0.096879408	1.28247162756678e-06\\
0.121901578	8.20771470526284e-07\\
0.146778406	5.69980890831017e-07\\
0.174162419	4.18765343212946e-07\\
0.200223271	3.20621943063987e-07\\
0.219093056	2.53335700236259e-07\\
0.241223774	2.05206623579729e-07\\
0.266694441	1.69596647955301e-07\\
0.290595881	1.42512437372732e-07\\
0.312285343	1.21434612704836e-07\\
0.349817584	1.04710074419824e-07\\
0.387253813	9.12175920507468e-08\\
0.387465405	8.01749937247511e-08\\
0.411445785	7.10231437935932e-08\\
0.432960668	6.33538273654467e-08\\
0.466032369	5.68632762858819e-08\\
0.483803531	5.13217256068615e-08\\
};
\addlegendentry{\texttt{g15}}

\end{axis}
\end{tikzpicture}%
}
\caption{Errors of the numerical solutions with smooth initial data.\\
The dashed line indicates order 2.
}
\label{fig:sineKG_th16}
\end{figure}
\begin{figure}[htp]
\scalebox{.9}{
%
%
\definecolor{color_rs21}{rgb}{0.30100,0.74500,0.93300}%
\definecolor{color_lsz}{rgb}{0.85000,0.32500,0.09800}%
\definecolor{color_ss68}{rgb}{0.92900,0.69400,0.12500}%
\definecolor{color_d79}{rgb}{0.49400,0.18400,0.55600}%
\definecolor{color_g15}{rgb}{0.46600,0.67400,0.18800}%
\definecolor{color_hl21}{rgb}{0.08240,0.26270,0.37650}%
\begin{tikzpicture}

\begin{axis}[%
width=2.300in,
height=3.100in,
at={(1.011in,0.642in)},
scale only axis,
xmode=log,
xmin=0.0001,
xmax=0.1,
xminorticks=true,
xlabel style={font=\color{white!15!black}},
xlabel={$\tau$},
ymode=log,
ymin=1e-08,
ymax=1e3,
yminorticks=true,
ylabel style={font=\color{white!15!black}},
ylabel={{\small $H^1(\Omega) \times L^2(\Omega)$ error}},
axis background/.style={fill=white},
xmajorgrids,
ymajorgrids,
legend style={at={(0.985,0.625)}, anchor=south east, legend cell align=left, align=left, draw=white!15!black}
]

\addplot [color=color_hl21, mark=asterisk, mark options={solid, color_hl21,thick,scale=1.5}]
  table[row sep=crcr]{%
0.002	137.212811254545\\
0.001	137.182392380214\\
0.000666666666666667	136.026534310212\\
0.0005	136.508253996708\\
0.0004	134.323454880322\\
0.000333333333333333	134.245423970011\\
0.000285714285714286	133.632536784789\\
0.00025	134.951667674565\\
0.000222222222222222	135.113725278869\\
0.0002	134.587189765165\\
0.000181818181818182	134.252400771763\\
0.000166666666666667	133.094294210245\\
0.000153846153846154	134.611412061962\\
0.000142857142857143	133.430051530192\\
0.000133333333333333	133.58083062073\\
0.000125	131.748733870146\\
0.000117647058823529	132.119655092046\\
0.000111111111111111	134.267820900791\\
0.000105263157894737	131.978688051394\\
0.0001	131.907098263371\\
};
\addlegendentry{\texttt{hl21}}

\addplot [color=color_rs21, mark=o, mark options={solid, color_rs21,thick,scale=1.5}]
  table[row sep=crcr]{%
0.02	0.0101773196653171\\
0.01	0.00737025549665097\\
0.00666666666666667	0.00571855250193662\\
0.005	0.00636263193373351\\
0.004	0.0034517876864368\\
0.00333333333333333	0.00264079941695176\\
0.00285714285714286	0.00219101848074782\\
0.0025	0.00317681043086918\\
0.00222222222222222	0.00329425100468963\\
0.002	0.00146909331386035\\
0.00181818181818182	0.000871681930580708\\
0.00166666666666667	0.000633066542321822\\
0.00153846153846154	0.000973738658002708\\
0.00142857142857143	0.0010652667270076\\
0.00133333333333333	0.000735957416348494\\
0.00125	0.00133809772138613\\
0.00117647058823529	0.00054477221853179\\
0.00111111111111111	0.00121016538430448\\
0.00105263157894737	0.0010154854312548\\
0.001	0.000507284871555854\\
};
\addlegendentry{\texttt{rs21}}

\addplot [color=color_lsz, mark=star, mark options={solid, color_lsz,thick,scale=1.5}]
  table[row sep=crcr]{%
0.02	2.65805853187462e-05\\
0.01	6.79754365032311e-06\\
0.00666666666666667	3.06455065744109e-06\\
0.005	1.74629836269642e-06\\
0.004	1.11342507265609e-06\\
0.00333333333333333	7.70869734541034e-07\\
0.00285714285714286	5.68800344700601e-07\\
0.0025	4.38363995484262e-07\\
0.00222222222222222	3.44349979998181e-07\\
0.002	2.79700019329206e-07\\
0.00181818181818182	2.40594814235344e-07\\
0.00166666666666667	2.26704478487547e-07\\
0.00153846153846154	2.23363047560301e-07\\
0.00142857142857143	2.00472180746656e-07\\
0.00133333333333333	1.51996805233482e-07\\
0.00125	1.35225463583001e-07\\
0.00117647058823529	1.42542920372057e-07\\
0.00111111111111111	1.22717643446559e-07\\
0.00105263157894737	9.66084684856104e-08\\
0.001	7.69929296895122e-08\\
};
\addlegendentry{\texttt{c\_lie}}

\addplot [color=color_ss68, mark=square, mark options={solid, color_ss68,thick,scale=1.5}]
  table[row sep=crcr]{%
0.02	0.01017746722646\\
0.01	0.00737027399070917\\
0.00666666666666667	0.00571857833051722\\
0.005	0.00636262115640919\\
0.004	0.00345179430467914\\
0.00333333333333333	0.00264080363987067\\
0.00285714285714286	0.00219101705202179\\
0.0025	0.00317681039789276\\
0.00222222222222222	0.00329425246523131\\
0.002	0.00146909384918029\\
0.00181818181818182	0.000871681995392831\\
0.00166666666666667	0.000633066001190735\\
0.00153846153846154	0.000973738865702768\\
0.00142857142857143	0.00106526735329974\\
0.00133333333333333	0.000735957669848174\\
0.00125	0.00133809763870758\\
0.00117647058823529	0.000544771935743685\\
0.00111111111111111	0.00121016475406997\\
0.00105263157894737	0.00101548538959184\\
0.001	0.000507284795529641\\
};
\addlegendentry{\texttt{ss68}}

\addplot [color=color_d79, mark=x, mark options={solid, color_d79,thick,scale=1.5}]
  table[row sep=crcr]{%
0.02	0.0101772920298114\\
0.01	0.00737025271361784\\
0.00666666666666667	0.00571855210910112\\
0.005	0.00636263257224432\\
0.004	0.00345179054241238\\
0.00333333333333333	0.00264079921315185\\
0.00285714285714286	0.00219101997163739\\
0.0025	0.0031768108065212\\
0.00222222222222222	0.00329425083028516\\
0.002	0.0014690931948634\\
0.00181818181818182	0.000871683353990343\\
0.00166666666666667	0.000633068482858608\\
0.00153846153846154	0.000973737989416861\\
0.00142857142857143	0.0010652684427361\\
0.00133333333333333	0.00073596104963007\\
0.00125	0.00133810124433513\\
0.00117647058823529	0.000544773735306634\\
0.00111111111111111	0.00121016563348667\\
0.00105263157894737	0.00101548810481167\\
0.001	0.000507287585577348\\
};
\addlegendentry{\texttt{d79}}

\addplot [color=color_g15, mark=x, mark options={solid, color_g15,thick,scale=1.5}]
  table[row sep=crcr]{%
0.02	0.0699616739464838\\
0.01	0.0479169069522781\\
0.00666666666666667	0.0385883985735841\\
0.005	0.0326615885099093\\
0.004	0.0290930839024245\\
0.00333333333333333	0.0262433326540627\\
0.00285714285714286	0.0240020984592913\\
0.0025	0.0227607279793295\\
0.00222222222222222	0.0214623155186661\\
0.002	0.0202254829316616\\
0.00181818181818182	0.0188412750710067\\
0.00166666666666667	0.0179889992233129\\
0.00153846153846154	0.0170731579536167\\
0.00142857142857143	0.0165967633440258\\
0.00133333333333333	0.0158856972240363\\
0.00125	0.0153352046969824\\
0.00117647058823529	0.0147450598303862\\
0.00111111111111111	0.0142708458558187\\
0.00105263157894737	0.0137254897550043\\
0.001	0.0132480186198876\\
};
\addlegendentry{\texttt{g15}}

\addplot [color=black, dashdotted]
  table[row sep=crcr]{%
1.60e-01	  5.596933915718704e-01\\
1.95e-05    6.832194721336308e-05\\
};

\addplot [color=black, dashed]
  table[row sep=crcr]{%
0.16	1.701157460399757e-03\\
0.0001	6.64514632968655e-10\\
};

\end{axis}
\end{tikzpicture}%
%
%
\definecolor{color_rs21}{rgb}{0.30100,0.74500,0.93300}%
\definecolor{color_lsz}{rgb}{0.85000,0.32500,0.09800}%
\definecolor{color_ss68}{rgb}{0.92900,0.69400,0.12500}%
\definecolor{color_d79}{rgb}{0.49400,0.18400,0.55600}%
\definecolor{color_g15}{rgb}{0.46600,0.67400,0.18800}%
\definecolor{color_hl21}{rgb}{0.08240,0.26270,0.37650}%
\begin{tikzpicture}

\begin{axis}[%
width=2.300in,
height=3.100in,
at={(1.011in,0.642in)},
scale only axis,
xmode=log,
xmin=0.01,
xmax=10,
xminorticks=true,
xlabel style={font=\color{white!15!black}},
xlabel={wall-clock time},
ymode=log,
ymin=1e-08,
ymax=1e3,
yminorticks=true,
ylabel style={font=\color{white!15!black}},
ylabel={{\small $H^1(\Omega) \times L^2(\Omega)$ error}},
axis background/.style={fill=white},
xmajorgrids,
ymajorgrids,
legend style={at={(0.65,0.03)}, anchor=south west, legend cell align=left, align=left, draw=white!15!black}
]

\addplot [color=color_hl21, mark=asterisk, mark options={solid, color_hl21,thick,scale=1.5}]
  table[row sep=crcr]{%
0.135386947	137.212811254545\\
0.260818653	137.182392380214\\
0.389982566	136.026534310212\\
0.518308531	136.508253996708\\
0.65134817	134.323454880322\\
0.796103721	134.245423970011\\
0.907718803	133.632536784789\\
1.042348147	134.951667674565\\
1.168175996	135.113725278869\\
1.29934286	134.587189765165\\
1.433967629	134.252400771763\\
1.556533189	133.094294210245\\
1.687514826	134.611412061962\\
1.841382258	133.430051530192\\
2.005947305	133.58083062073\\
2.085132534	131.748733870146\\
2.216541072	132.119655092046\\
2.348577531	134.267820900791\\
2.483788756	131.978688051394\\
2.603918126	131.907098263371\\
};
\addlegendentry{\texttt{hl21}}

\addplot [color=color_rs21, mark=o, mark options={solid, color_rs21,thick,scale=1.5}]
  table[row sep=crcr]{%
0.026553785	0.0101773196653171\\
0.046568363	0.00737025549665097\\
0.068526292	0.00571855250193662\\
0.089597681	0.00636263193373351\\
0.113134289	0.0034517876864368\\
0.135388053	0.00264079941695176\\
0.156061865	0.00219101848074782\\
0.178599873	0.00317681043086918\\
0.203626517	0.00329425100468963\\
0.22338778	0.00146909331386035\\
0.244491574	0.000871681930580708\\
0.264942098	0.000633066542321822\\
0.289643416	0.000973738658002708\\
0.313647241	0.0010652667270076\\
0.33349677	0.000735957416348494\\
0.36114893	0.00133809772138613\\
0.379513502	0.00054477221853179\\
0.397190259	0.00121016538430448\\
0.425170673	0.0010154854312548\\
0.446952444	0.000507284871555854\\
};
\addlegendentry{\texttt{rs21}}

\addplot [color=color_lsz, mark=star, mark options={solid, color_lsz,thick,scale=1.5}]
  table[row sep=crcr]{%
0.026077905	2.65805853187462e-05\\
0.055310989	6.79754365032311e-06\\
0.078895623	3.06455065744109e-06\\
0.102197289	1.74629836269642e-06\\
0.136956838	1.11342507265609e-06\\
0.151327121	7.70869734541034e-07\\
0.176564775	5.68800344700601e-07\\
0.200394176	4.38363995484262e-07\\
0.226132376	3.44349979998181e-07\\
0.254337925	2.79700019329206e-07\\
0.281079829	2.40594814235344e-07\\
0.302382155	2.26704478487547e-07\\
0.33225116	2.23363047560301e-07\\
0.353286415	2.00472180746656e-07\\
0.375957436	1.51996805233482e-07\\
0.405498155	1.35225463583001e-07\\
0.431469404	1.42542920372057e-07\\
0.459666318	1.22717643446559e-07\\
0.477592602	9.66084684856104e-08\\
0.504613626	7.69929296895122e-08\\
};
\addlegendentry{\texttt{c\_lie}}

\addplot [color=color_ss68, mark=square, mark options={solid, color_ss68,thick,scale=1.5}]
  table[row sep=crcr]{%
0.02207565	0.01017746722646\\
0.024994238	0.00737027399070917\\
0.037963878	0.00571857833051722\\
0.050692768	0.00636262115640919\\
0.061658787	0.00345179430467914\\
0.074737873	0.00264080363987067\\
0.087512087	0.00219101705202179\\
0.098456505	0.00317681039789276\\
0.110663831	0.00329425246523131\\
0.122527793	0.00146909384918029\\
0.137476061	0.000871681995392831\\
0.149846392	0.000633066001190735\\
0.159259143	0.000973738865702768\\
0.171352049	0.00106526735329974\\
0.182170879	0.000735957669848174\\
0.194185188	0.00133809763870758\\
0.209094901	0.000544771935743685\\
0.222468492	0.00121016475406997\\
0.234787066	0.00101548538959184\\
0.24458948	0.000507284795529641\\
};
\addlegendentry{\texttt{ss68}}

\addplot [color=color_d79, mark=x, mark options={solid, color_d79,thick,scale=1.5}]
  table[row sep=crcr]{%
0.033837766	0.0101772920298114\\
0.051654546	0.00737025271361784\\
0.075248728	0.00571855210910112\\
0.098418417	0.00636263257224432\\
0.124036825	0.00345179054241238\\
0.154409549	0.00264079921315185\\
0.1729442	0.00219101997163739\\
0.200709483	0.0031768108065212\\
0.220818817	0.00329425083028516\\
0.247324065	0.0014690931948634\\
0.271984328	0.000871683353990343\\
0.297110471	0.000633068482858608\\
0.324638823	0.000973737989416861\\
0.349033069	0.0010652684427361\\
0.374802979	0.00073596104963007\\
0.396466182	0.00133810124433513\\
0.420859712	0.000544773735306634\\
0.448270721	0.00121016563348667\\
0.472811531	0.00101548810481167\\
0.499720741	0.000507287585577348\\
};
\addlegendentry{\texttt{d79}}

\addplot [color=color_g15, mark=x, mark options={solid, color_g15,thick,scale=1.5}]
  table[row sep=crcr]{%
0.033074677	0.0699616739464838\\
0.050112709	0.0479169069522781\\
0.074724605	0.0385883985735841\\
0.100650134	0.0326615885099093\\
0.124689511	0.0290930839024245\\
0.150696627	0.0262433326540627\\
0.176379455	0.0240020984592913\\
0.199663869	0.0227607279793295\\
0.226610411	0.0214623155186661\\
0.245306955	0.0202254829316616\\
0.278171314	0.0188412750710067\\
0.300027474	0.0179889992233129\\
0.324284955	0.0170731579536167\\
0.348781956	0.0165967633440258\\
0.370579391	0.0158856972240363\\
0.395124893	0.0153352046969824\\
0.418732037	0.0147450598303862\\
0.447019947	0.0142708458558187\\
0.474432169	0.0137254897550043\\
0.499226623	0.0132480186198876\\
};
\addlegendentry{\texttt{g15}}

\end{axis}
\end{tikzpicture}%
}
\caption{Errors of the numerical solutions with $H^1(\Omega)\times L^2(\Omega)$ initial data.\\
The dashed lines indicate orders 1 and 2, respectively.}
\label{fig:sineKG_th1}
\end{figure}
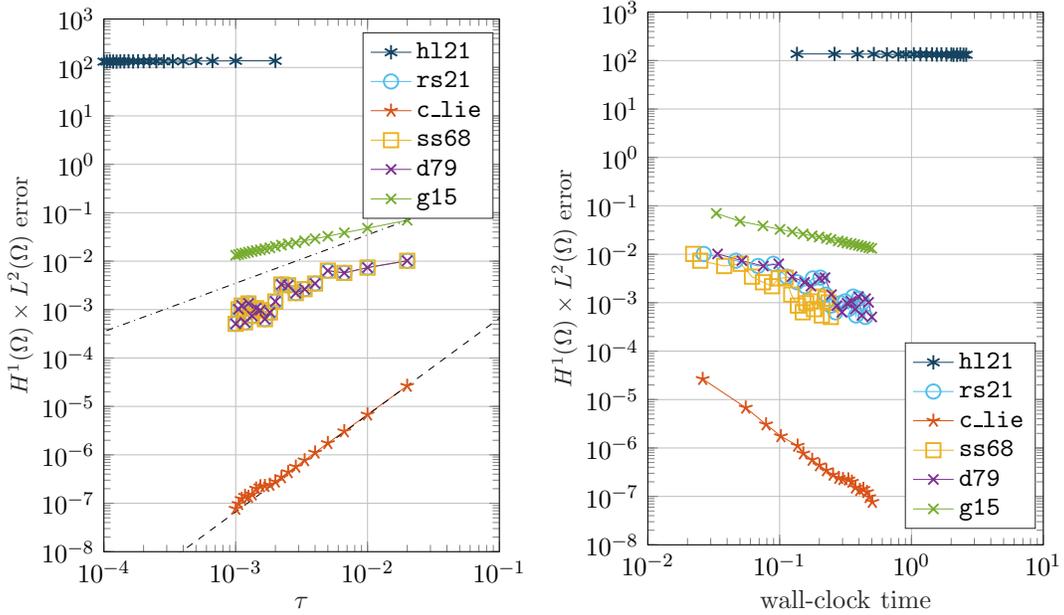
%

%
%
%

\section*{Funding}

The work of Buyang Li is partially supported by the Hong Kong Research Grants Council (General Research Fund, project no. 15300519).
Franco Zivcovich and Katharina Schratz have received funding from the European Research Council (ERC) under the European Union’s Horizon 2020 research and innovation programme (grant agreement No. 850941).


\end{document}